\documentclass[10pt,a4paper]{article}

\usepackage{graphicx}

\usepackage{amsthm,amssymb,mathrsfs,setspace,amsmath}
\usepackage{bigints}
\usepackage[a4paper,bindingoffset=0.2in,
left=0.5in,right=0.5in,
top=1in,bottom=1in,footskip=.25in]{geometry}
\title{A note on continuous fractional wavelet transform in $\mathbb{R}^n$}
\author{Amit K. Verma$^a$\thanks{Corresponding author email: akverma@iitp.ac.in}, Bivek Gupta$^b$\\{\small\textit{$^{a,b}$ Department of Mathematics, IIT Patna, Bihta, Patna 801103.}}}
\theoremstyle{definition}
\newtheorem{defn}{Definition}[section]
\newtheorem{theorem}{Theorem}[section]
\newtheorem{corollary}{Corollary}[section]
\newtheorem{lemma}{Lemma}[section]
\newtheorem{note}{Note}[theorem]
\newtheorem{prop}{Proposition}
\begin{document}
\maketitle
\begin{abstract}
In this paper, we have studied continuous fractional wavelet transform (CFrWT) in $n$-dimensional Euclidean space $\mathbb{R}^n$  with scaling parameter $\boldsymbol a=(a_{1},a_{2},\ldots,a_{n}),$ such that none of $a_{i}'s$ are zero. Necessary and sufficient condition for the wavelet admissibility of a function is established with the help of fractional convolution. Inner product relation and reconstruction formula for the CFrWT depending on two wavelets are obtained along with the reproducing kernel function, involving two wavelets, for the image space of CFrWT. Heisenberg's  uncertainty inequality and Local uncertainty inequality for the CFrWT are obtained. Finally, boundedness of the transform on the Morrey space $L^{1,\nu}_{M}(\mathbb{R}^n)$ and the estimate of $L^{1,\nu}_{M}(\mathbb{R}^n)$-distance of the CFrWT of two argument functions with respect to different wavelets are discussed.\vspace{0.2cm}
\end{abstract}
\section{Introduction}
The term fractional Fourier transform (FrFT) was first coined by Namias (\cite{namias1980fractional}) in the year 1980. The flaw in the factor preceding the exponential in the definition of the kernel was rectified by McBride et al. (\cite{mcbride1987namias}) in 1987. Almeida (\cite{almeida1994fractional}), Bultheel et al. (\cite{bultheel2005introduction}) studied the properties of the kernel involved in the FrFT. The $n$-dimensional FrFT (\cite{toksoy2015function,upadhyay2014characterization,zayed2018two}) is the tensor product of $n$ copies of the one-dimensional FrFT. Zayed (\cite{zayed2018two}) came up with the new two-dimensional FrFT which is not the tensor product of two one-dimensional FrFT. Dai et al. (\cite{dai2017new}) introduced the definition of CFrWT which is more general to the transform defined by Shi et al. (\cite{shi2012novel}). It is also a generalization of classical wavelet transform. It represents a signal in time-fractional-frequency plane and provide the time domain and fractional frequency domain information simultaneously, which is not possible for the FrFT. Bahri et al. (\cite{bahri2017logarithmic}), Prasad et al. (\cite{prasad2015continuous}), Prasad et al. (\cite{prasad2014generalized}) studied CFrWT in one-dimension with the scaling parameter restricted over the set of positive reals. They also discussed the properties of the CFrWT on generalized Sobolev space which is the natural extension of the properties of classical wavelet transform  studied in \cite{chuong2002integral, pathak2009wavelet,rieder1990wavelet,triintegral}.\\
The spread of a function in the time domain and in the frequency domain are inversely related, i.e., more the function is localized in the time domain the more will be the function spread out in the frequency domain. This fact is supported by the Heisenberg's uncertainty inequality (\cite{cowling1984bandwidth,folland1997uncertainty}), where the measure of spread of a function is given by the standard deviation of the square of the absolute value of the function. Apart from being widely spread in the frequency domain, when the function is localized in time, the function in the frequency domain cannot be highly localized at any points, this is interpreted by Local uncertainty inequality (\cite{folland1997uncertainty,price1983inequalities,price1987sharp}). The Heisenberg's uncertainty inequality in two fractional Fourier transform domain can be found in \cite{aytur1995non,guanlei2009logarithmic,ozaktas1995fractional,zhao2009uncertainty}. From an idea of Singer (\cite{singer1999uncertainty}) in the case of wavelet transform, Wilczok (\cite{wilczok2000new}) presented the uncertainty principle for continuous wavelet and Gabor transform.\\
The main objectives of this paper are as follows:
\begin{enumerate}
\item To extend the theory of CFrWT, studied by \cite{bahri2017logarithmic},\cite{prasad2015continuous},\cite{prasad2014generalized} with scaling parameter restricted to the positive reals, to $n$-dimensional euclidean space with the scaling parameter $\boldsymbol a=(a_{1},a_{2},\ldots,a_{n})\in\mathbb{R}^n$ with $a_{i}\neq 0$ for all $i=1,2,\ldots,n.$
\item To establish the Heisenberg's uncertainty inequality and Local uncertainty inequality for CFrWT.
\item To study the properties of the CFrWT on Morrey Space. 
\end{enumerate}
This paper is organized as follows: All the results proved in this paper are in $n$ dimension. In section 2, we have given the basic definitions and properties of $n$ dimensional FrFT, which is infact the tensor product of $n$ copies of one-dimensional FrFT. We have obtained the necessary  and sufficient condition for the wavelet admissibility of a function. In section 3, we have discussed CFrWT with scaling parameter in $\mathbb{R}^n$ none of whose components are zero. Properties like inner product relation and reconstruction formula for the transform depending on two wavelets are derived along with the reproducing kernel function, involving two wavelets, for the image space of CFrWT. In section 4 we have obtained two important inequalities viz., Heisenberg's uncertainty inequality and Local uncertainty inequality for CFrWT. Finally, in section 5, we have studied the CFrWT on Morrey space $L^{1,\nu}_{M}(\mathbb{R}^n)$ along with its approximation properties.
\section{Preliminaries}
The following are the notations that we will be using throughout the paper:\\
$\mathbb{R}^n$  denotes the $n$-dimensional Euclidean space with the euclidean norm $\|\cdot\|,$ i.e., for $\boldsymbol x=(x_{1},x_{2},\ldots,x_{n})\in\mathbb{R}^n$, $$\|\boldsymbol x\|=\sqrt{\sum_{i=1}^n x^2_{i}}.$$
We define $|\boldsymbol x|_{p}=|x_{1}x_{2}\ldots x_{n}|$ and $\mathbb{R}^n_{0}=\{\boldsymbol x\in\mathbb{R}^n:|\boldsymbol x|_{p}\neq 0\}.$ For $\boldsymbol x=(x_{1},x_{2},\ldots,x_{n}),~\boldsymbol y=(y_{1},y_{2},\ldots,y_{n})\in\mathbb{R}^n,$ $\boldsymbol x+\boldsymbol y=(x_{1}+y_{1},x_{2}+y_{2},\ldots,x_{n}+y_{n}),~\boldsymbol x\boldsymbol y=(x_{1}y_{1},x_{2}y_{2},\ldots,x_{n}y_{n}).$ If further $\boldsymbol y\in\mathbb{R}^n_{0}$, then $\frac{\boldsymbol x}{\boldsymbol y}=(\frac{x_{1}}{y_{1}},\frac{x_{2}}{y_{2}},\ldots,\frac{x_{n}}{y_{n}}).$
\begin{defn}
For $1\leq p<\infty,$ $L^p(\mathbb{R}^n)$ is a Banach space of all complex valued measurable function defined on $\mathbb{R}^n$ such that
$$\int_{\mathbb{R}^n}|f(\boldsymbol t)|^pd\boldsymbol t<\infty,$$ with norm defined by
$$\|f\|_{L^p(\mathbb{R}^n)}=\left(\int_{\mathbb{R}^n}|f(\boldsymbol t)|^p d\boldsymbol t \right)^{\frac{1}{p}}.$$
In particular, $L^2(\mathbb{R}^n)$ is a Hilbert space, where the inner product inducing the norm is given by 
$$\langle f,g\rangle_{L^2(\mathbb{R}^n)}=\int_{\mathbb{R}^n}f(\boldsymbol t)\overline{g(\boldsymbol t)}d\boldsymbol t.$$ 
\end{defn}
\begin{defn}
The Schwartz class, denoted by $\mathcal{S}(\mathbb{R}^n),$ is defined as the set of all $f\in C^\infty(\mathbb{R}^n)$ such that for all multi-indices $\boldsymbol r,\boldsymbol s$
$$\sup_{\boldsymbol x\in\mathbb{R}^n}\left|{\boldsymbol x}^{\boldsymbol r}(D^{\boldsymbol s}f)(\boldsymbol x)\right|<\infty,$$
where, $\boldsymbol x^{\boldsymbol r}=x_{1}^{r_{1}}x_{2}^{r_{2}}\ldots x_{n}^{r_{n}}$ and $D^{\boldsymbol s}=\frac{\partial^{\boldsymbol s}}{\partial x_{1}^{s_{1}}\partial x_{2}^{s_{2}}\ldots \partial x_{n}^{s_{n}}},$ if $\boldsymbol x=(x_{1},x_{2},\ldots, x_{n}),\boldsymbol r=(r_{1},r_{2},\ldots,r_{n})$ and $\boldsymbol s=(s_{1},s_{2},\ldots,s_{n}).$
\end{defn}
\begin{defn}
Let $A$ be an arbitrary set and $X$ be a Hilbert space of complex valued functions defined on $A$ with the inner product $\langle\cdot,\cdot\rangle_{X}.$  Then a complex valued function $K$ defined on $A\times A$ is called a reproducing kernel of X if it satisfies the following condition:\\
For any fixed $q\in A,$ we have $K(\cdot,q)$ is in $X$ and $f(q)=\langle f(\cdot),K(\cdot,q)\rangle_{X}$ for all $f\in X$. 
\end{defn}
\subsection{Fractional Fourier Transform (FrFT) of functions in $L^1(\mathbb{R}^n)$ and its properties}
\begin{defn}\label{P1defnK}
The fractional Fourier transform (FrFT), with real parameter $\alpha,$ of $f\in L^1(\mathbb{R}^n)$ is denoted by $\mathfrak{F}_{\alpha}f$ and defined by
$$(\mathfrak{F}_\alpha f)(\boldsymbol\xi)=\int_{\mathbb{R}^n}f(\boldsymbol t)K_{\alpha}(\boldsymbol t,\boldsymbol\xi)d\boldsymbol t,\ \boldsymbol\xi\in\mathbb{R}^n$$
where
\begin{equation}
K_{\alpha}(\boldsymbol t,\boldsymbol\xi)=
\begin{cases}
C_{\alpha}e^{\frac{i}{2}(\|\boldsymbol t\|^2+\|\boldsymbol\xi\|^2)\cot\alpha-i\langle\boldsymbol t,\boldsymbol\xi\rangle\csc\alpha},~~~\mbox{if}~\alpha\neq m\pi\\
\delta(\boldsymbol t-\boldsymbol\xi),~~~~~~~~~~~~~~~~~~~~~~~~~~~~~~~~\mbox{if}~\alpha=2m\pi\\
\delta(\boldsymbol t+\boldsymbol\xi),~~~~~~~~~~~~~~~~~~~~~~~~~~~~~~~~\mbox{if}~\alpha=(2m+1)\pi,m\in\mathbb{Z},
\end{cases}
\end{equation}
$\delta$ is a Dirac delta function in $n$-dimensions and 
$$C_{\alpha}=\left(\frac{1-i\cot\alpha}{2\pi}\right)^{\frac{n}{2}}.$$
\end{defn}
Note that the FrFT reduces to the classical Fourier transform in $n$-dimensions if $\alpha=\frac{\pi}{2},$ i.e., $\mathfrak{F}_{\frac{\pi}{2}}=\mathfrak{F}.$\\
We mention below some properties of the kernel that will be used shortly,
\begin{description}
\item[(K1)] $K_{\alpha}(\boldsymbol t,\boldsymbol\xi)=K_{\alpha}(\boldsymbol\xi,\boldsymbol t),$
\item[(K2)] $\overline{K_{\alpha}(\boldsymbol t,\boldsymbol\xi)}=K_{-\alpha}(\boldsymbol t,\boldsymbol\xi),$
\item[(K3)] $\displaystyle\int_{\mathbb{R}^n}K_{\alpha}(\boldsymbol t,\boldsymbol\xi)K_{\beta}(\boldsymbol\xi,\boldsymbol u)d\boldsymbol\xi=K_{\alpha+\beta}(\boldsymbol t,\boldsymbol u),$
\item[(K4)] $K_{\alpha}$ is continuous for all $\alpha\neq m\pi,$ and in the generalized function sense it is even continuous for $\alpha=m\pi,~m\in\mathbb{Z}.$
\end{description}
\begin{prop}\label{P1prop1}
If  $f\in L^1(\mathbb{R}^n)$, then $$(\mathfrak{F}_\alpha f)(\boldsymbol\xi)=(2\pi)^\frac{n}{2}C_\alpha e^{\frac{i}{2}\|\boldsymbol\xi\|^2\cot\alpha}\left(\mathfrak{F}\left\{e^{\frac{i}{2}\|\boldsymbol t\|^2\cot\alpha}f(\boldsymbol t)\right\}\right)\left(\frac{\boldsymbol\xi}{\sin\alpha}\right),\boldsymbol\xi\in\mathbb{R}^n.$$
\end{prop}
\begin{proof}
Let $\boldsymbol\xi\in\mathbb{R}^n$, then from the definition, we have
\begin{eqnarray*}
(\mathfrak{F}_{\alpha}f)(\boldsymbol\xi)&=&\int_{\mathbb{R}^n}C_{\alpha}e^{\frac{i}{2}\left(\|\boldsymbol t\|^2+\|\boldsymbol\xi\|^2\right)\cot\alpha-i\left\langle \boldsymbol t,\frac{\boldsymbol\xi}{\sin\alpha}\right\rangle}f(\boldsymbol t)d\boldsymbol t\\
&=&(2\pi)^{\frac{n}{2}} C_{\alpha}e^{\frac{i}{2}\|\boldsymbol\xi\|^2\cot\alpha}\frac{1}{(2\pi)^{\frac{n}{2}}}\int_{\mathbb{R}^n}\left\{e^{\frac{i}{2}\|\boldsymbol t\|^2\cot\alpha}f(\boldsymbol t)\right\}e^{-i\left\langle \boldsymbol t,\frac{\boldsymbol\xi}{\sin\alpha}\right\rangle}d\boldsymbol t
\end{eqnarray*}
Thus, the proposition follows.
\end{proof}
\begin{prop}\label{P1prop2}
If $f\in L^{1}(\mathbb{R}^n)\cap L^{2}(\mathbb{R}^n)$, then $\mathfrak{F_\alpha}f\in L^{2}(\mathbb{R}^n)$.
\end{prop}          
\begin{proof}
This follows from proposition \ref{P1prop1}.
\end{proof}                                         
\begin{prop}\label{P1prop3}
If $f\in\mathcal{S}(\mathbb{R}^n)$, then ${\mathfrak{F}}_\alpha f\in\mathcal{S}(\mathbb{R}^n)$.
\end{prop}
\begin{proof}
This follows immediately from proposition \ref{P1prop1}.
\end{proof}
\begin{prop}\label{P1prop4}
If $f\in\mathcal{S}(\mathbb{R}^n)$, then $\mathfrak{F}_\alpha(\mathfrak{F}_{-\alpha}f)\in\mathcal{S}(\mathbb{R}^n)$.
\end{prop}
\begin{proof}
It can be proved using proposition \ref{P1prop3}.
\end{proof}
\begin{prop}\label{P1prop5}
If $f\in L^1(\mathbb{R}^n)$, then $\overline{\mathfrak{F}_\alpha f}=\mathfrak{F}_{-\alpha}\bar{f},$ where bar denotes the complex conjugate.
\end{prop}
\begin{proof}
This follows from the property (K2) of the kernel $K_{\alpha}(\boldsymbol t,\boldsymbol\xi)$ as given in the definition (\ref{P1defnK}).
\end{proof}
\begin{theorem}\label{P1theo1.3}
If $f\in L^1(\mathbb{R}^n)$, then $\mathfrak{F}_\alpha f\in L^\infty(\mathbb{R}^n)$. Moreover, $\mathfrak{F}_\alpha f$ is continuous (hence measurable) on $\mathbb{R}^n$. 
\end{theorem}
\begin{proof}
For $\boldsymbol\xi\in\mathbb{R}^n,$ we have
$$|(\mathfrak{F}_\alpha f)(\boldsymbol\xi)|\leq \int_{\mathbb{R}^n}|K_\alpha(\boldsymbol t,\boldsymbol\xi)f(\boldsymbol t)|d\boldsymbol t\leq |C_\alpha|\int_{\mathbb{R}^n} |f(\boldsymbol t)|d\boldsymbol t=|C_\alpha|\|f\|_{L^1(\mathbb{R}^n)}.$$
Thus, it follows that $\mathfrak{F}_{\alpha}f\in L^{\infty}(\mathbb{R}^n)$.\\
Let $\boldsymbol \xi\in {R}^n$ and $\{{\boldsymbol \xi}_k\}$ be a sequence in $\mathbb{R}^n$ converging to $\boldsymbol\xi$. Then
\begin{equation}\label{P1eqnA} 
 \lim_{k\to\infty}f(\boldsymbol t)K_\alpha(\boldsymbol t,{\boldsymbol\xi}_k)=f(\boldsymbol t)K_\alpha(\boldsymbol t,\boldsymbol\xi).
 \end{equation}
Using equation (\ref{P1eqnA}) and the dominated convergence theorem, we have
 $$\lim_{k\to\infty}(\mathfrak{F}_\alpha f)({\boldsymbol\xi}_k)=(\mathfrak{F}_\alpha f)(\boldsymbol\xi).$$
So, $f$ is continuous on $\mathbb{R}^n$.
\end{proof}
\begin{theorem}\label{P1theo1.1}
(Paresval's Theorem) Let $f,g\in L^{1}(\mathbb{R}^n)\cap L^{2}(\mathbb{R}^n)$, then ${\langle f,g\rangle}_{L^{2}(\mathbb{R}^n)}={\langle\mathfrak{F_\alpha}f,\mathfrak{F_\alpha}g \rangle}_{L^{2}(\mathbb{R}^n)}$. Particularly, $\|f\|_{L^{2}(\mathbb{R}^n)}=\|\mathfrak{F_\alpha}f\|_{L^{2}(\mathbb{R}^n)},$ for all $f\in L^{1}(\mathbb{R}^n)\cap L^{2}(\mathbb{R}^n)$.
\end{theorem}
\begin{proof}
We have,
\begin{eqnarray*}
{\langle\mathfrak{F_\alpha}f,\mathfrak{F_\alpha}g \rangle}_{L^{2}(\mathbb{R}^n)}&=
&\int_{\mathbb{R}^n} (\mathfrak{F_\alpha}f)(\boldsymbol{\xi})\overline{(\mathfrak{F_\alpha}g)(\boldsymbol{\xi})}d\boldsymbol{\xi}\\
&=&\int_{\mathbb{R}^n}(2\pi)^{n}C_{\alpha}e^{\frac{i}{2}{\|\boldsymbol{\xi}\|}^2\cot\alpha}\left(\mathfrak{F}\left\{e^{\frac{i}{2}{\|\cdot\|}^2\cot\alpha}f(\cdot)\right\}\right)\left(\frac{\boldsymbol{\xi}}{\sin\alpha}\right)\overline{C_{\alpha}}e^{-\frac{i}{2}{\|\boldsymbol{\xi}\|}^2\cot\alpha}\overline{\left(\mathfrak{F}\left\{e^{\frac{i}{2}{\|\cdot\|}^2\cot\alpha}g(\cdot)\right\}\right)\left(\frac{\boldsymbol{\xi}}{\sin\alpha}\right)}d\boldsymbol{\xi}\\
&=&(2\pi)^{n}\int_{\mathbb{R}^n}C_{\alpha}\overline{C_{\alpha}}\left(\mathfrak{F}\left\{e^{\frac{i}{2}{\|\cdot\|}^2\cot\alpha}f(\cdot)\right\}\right)\left(\frac{\boldsymbol{\xi}}{\sin\alpha}\right)\overline{\left(\mathfrak{F}\left\{e^{\frac{i}{2}{\|\cdot\|}^2\cot\alpha}g(\cdot)\right\}\right)\left(\frac{\boldsymbol{\xi}}{\sin\alpha}\right)}d\boldsymbol{\xi}\\
&=&(2\pi)^{n}|C_{\alpha}|^2\int_{\mathbb{R}^n}\left(\mathfrak{F}\left\{e^{\frac{i}{2}{\|\cdot\|}^2\cot\alpha}f(\cdot)\right\}\right)\left(\frac{\boldsymbol{\xi}}{\sin\alpha}\right)\overline{\left(\mathfrak{F}\left\{e^{\frac{i}{2}{\|\cdot\|}^2\cot\alpha}g(\cdot)\right\}\right)\left(\frac{\boldsymbol{\xi}}{\sin\alpha}\right)}d\boldsymbol{\xi}\\
&=&{\left\langle\mathfrak{F}\left\{e^{\frac{i}{2}{\|\cdot\|}^2\cot\alpha}f(\cdot)\right\},\mathfrak{F}\left\{e^{\frac{i}{2}{\|\cdot\|}^2\cot\alpha}g(\cdot)\right\}\right\rangle}_{L^{2}(\mathbb{R}^n)}.
\end{eqnarray*} 
Using the Parseval's formula for classical Fourier transform, we have 
\begin{eqnarray}\label{P1eqn2}
{\langle\mathfrak{F_\alpha}f,\mathfrak{F_\alpha}g \rangle}_{L^{2}(\mathbb{R}^n)}&=&{\left\langle e^{\frac{i}{2}{\|\cdot\|}^2\cot\alpha}f(\cdot),e^{\frac{i}{2}{\|\cdot\|}^2\cot\alpha}g(\cdot)\right\rangle}_{L^{2}(\mathbb{R}^n)}\notag\\
&=&{\langle f,g\rangle}_{L^{2}(\mathbb{R}^n)}.
\end{eqnarray}
Replacing $g$ by $f$ in equation (\ref{P1eqn2}), we obtain $\|f\|_{L^{2}(\mathbb{R}^n)}=\|\mathfrak{F_\alpha}f\|_{L^{2}(\mathbb{R}^n)}$. This completes the proof.
\end{proof}
\begin{defn}
For a complex valued measurable function $f$ on $\mathbb{R}^n,$ we define the following operators for some $
\boldsymbol\eta$ and ${\textbf{\em a}}$ in $\mathbb{R}^n$:
\begin{enumerate}
\item $(T_{\boldsymbol{\eta},\alpha}f)(\boldsymbol{t})=f(\boldsymbol{t}+\boldsymbol{\eta})e^{i\langle\boldsymbol{t},\boldsymbol{\eta}\rangle\cot{\alpha}},~\boldsymbol t\in\mathbb{R}^n,$
\item $(M_{\boldsymbol{\eta},\alpha}f)(\boldsymbol{t})=e^{i\langle\boldsymbol{t},\boldsymbol{\eta}\rangle\csc{\alpha}+\frac{i}{2}\|\boldsymbol\eta\|^2\cot\alpha} f(\boldsymbol{t}),~\boldsymbol t\in\mathbb{R}^n,$
\item $(D_{\textbf{\em a}}f)(\boldsymbol{t})=f({\textbf{\em a}}\boldsymbol t),~\boldsymbol t\in\mathbb{R}^n.$
\end{enumerate}
\end{defn}
\begin{prop}
If $f\in L^1(\mathbb{R}^n),$ then
\begin{enumerate}
\item  $(\mathfrak{F}_\alpha\{T_{\boldsymbol{\eta},\alpha}f\})(\boldsymbol{\xi})=(M_{\boldsymbol{-\eta},-\alpha}\{\mathfrak{F}_\alpha f\})(\boldsymbol{\xi}),~\boldsymbol\xi\in\mathbb{R}^n,$
\item $(\mathfrak{F}_\alpha\{M_{\boldsymbol{\eta},\alpha}f\})(\boldsymbol{\xi})=(T_{\boldsymbol{-\eta},-\alpha}\{\mathfrak{F}_\alpha f\})(\boldsymbol{\xi}),~\boldsymbol\xi\in\mathbb{R}^n,$
\item $(\mathfrak{F}_\alpha\{D_{\textbf{\em a}}f\})(\boldsymbol{\xi})=(\mathfrak{F}_\alpha f)({\textbf{\em a}}\boldsymbol \xi)$, where $\boldsymbol\xi\in\mathbb{R}^n,~ {\textbf{\em a}}=(a_{1},a_{2},...,a_{n})\in\mathbb{R}^n$ such that $|a_{i}|=1$ for all $i=1,2,...,n.$
\end{enumerate} 
\end{prop}
\begin{proof}
(1) Let $\boldsymbol\xi\in\mathbb{R}^n,$ then
\begin{eqnarray*}
(\mathfrak{F}_\alpha\{T_{\boldsymbol{\eta},\alpha}f\})(\boldsymbol{\xi})&=&\int_{\mathbb{R}^n}(T_{\boldsymbol{\eta},\alpha}f)(\boldsymbol{t})K_\alpha(\boldsymbol t,\boldsymbol\xi)d\boldsymbol{t}\\
&=&\int_{\mathbb{R}^n}f(\boldsymbol{t}+\boldsymbol{\eta})e^{i\langle\boldsymbol{t},\boldsymbol{\eta}\rangle\cot{\alpha}}C_{\alpha}e^{\frac{i}{2}(\|\boldsymbol{t}\|^2+\|\boldsymbol{\xi}\|^2)\cot\alpha-i\langle \boldsymbol t,\boldsymbol\xi\rangle\csc\alpha}d\boldsymbol{t}\\
&=& e^{-\frac{i}{2}\|\boldsymbol\eta\|^2\cot\alpha+i\langle\boldsymbol\eta,\boldsymbol\xi\rangle\csc\alpha}(\mathfrak{F}_\alpha f)(\boldsymbol\xi)\\
&=& (M_{\boldsymbol{-\eta},-\alpha}\{\mathfrak{F}_\alpha f\})(\boldsymbol{\xi}).\\
\end{eqnarray*}
(2) can be proved similarly.\\
(3) Let $\boldsymbol\xi\in\mathbb{R}^n,$ then
\begin{eqnarray*}
(\mathfrak{F}_\alpha\{D_{\textbf{\em a}}f\})(\boldsymbol{\xi})&=&\int_{\mathbb{R}^n}(D_{\textbf{\em a}}f)(\boldsymbol{t})K_\alpha(\boldsymbol t,\boldsymbol\xi)d\boldsymbol{t}\\
&=&\int_{\mathbb{R}^n}f({\textbf{\em a}}\boldsymbol{t})C_{\alpha}e^{\frac{i}{2}(\|\boldsymbol{t}\|^2+\|\boldsymbol{\xi}\|^2)\cot\alpha-i\langle \boldsymbol t,\boldsymbol\xi\rangle\csc\alpha}d\boldsymbol{t}\\
&=&\frac{1}{|{\textbf{\em a}}|_{p}}\int_{{R}^n}f(\boldsymbol{u})C_{\alpha}e^{\frac{i}{2}\left(\|{\frac{\boldsymbol{u}}{{\textbf{\em a}}}\|}^2+\|\boldsymbol{\xi}\|^2\right)\cot\alpha-i\left\langle \frac{\boldsymbol u}{{\textbf{\em a}}},\boldsymbol\xi\right\rangle\csc\alpha}d\boldsymbol{u}\\
&=&\frac{1}{|{\textbf{\em a}}|_{p}}\int_{{R}^n}f(\boldsymbol{u})C_{\alpha}e^{\frac{i}{2}({\|\boldsymbol{u}\|}^2+\|{\textbf{\em a}}\boldsymbol{\xi}\|^2)\cot\alpha-i\langle\boldsymbol u,{\textbf{\em a}}\boldsymbol\xi\rangle\csc\alpha}d\boldsymbol{u}\\
&=&(\mathfrak{F}_\alpha f)({\textbf{\em a}}\boldsymbol \xi),\ \mbox{since}\ |{\textbf{\em a}}|_{p}=1.
\end{eqnarray*}
\end{proof}
\subsection{FrFT of functions in $L^2(\mathbb{R}^n)$ and its properties}
 Let $f\in L^{2}(\mathbb{R}^n)$, then by the density property there exists a sequence $\{f_{k}\}$ in $L^{1}(\mathbb{R}^n)\bigcap L^{2}(\mathbb{R}^n)$ converging to $f$ in $L^{2}(\mathbb{R}^n)$.
Using theorem \ref{P1theo1.1}, we obtain 
 $${\|f_{p}-f_{q}\|}_{L^{2}(\mathbb{R}^n)}={\|\mathfrak{F}_{\alpha}f_{p}-\mathfrak{F}_\alpha f_{q}\|}_{L^{2}(\mathbb{R}^n)},$$ from which it follows that $\{\mathfrak{F}_\alpha f_{k}\}$ is a Cauchy sequence and hence by the Riesz-Fischer theorem we have $F\in L^{2}(\mathbb{R}^n)$ such that 
 $$\lim_{k\to\infty}\mathfrak{F}_\alpha f_{k}=F\ \mbox{in}\  L^{2}(\mathbb{R}^n).$$
 We claim that this $F$ is independent of the choice of the sequence $\{f_{k}\}$ in $L^{1}(\mathbb{R}^n)\bigcap L^{2}(\mathbb{R}^n)$.
 For, let $\{g_{k}\}$ be any sequence in $L^{1}(\mathbb{R}^n)\bigcap L^{2}(\mathbb{R}^n)$ converging to $f$ in $L^{2}(\mathbb{R}^n)$, then arguing similarly as above we have $F'$ in $L^{2}(\mathbb{R}^n)$ such that 
 $$\lim_{k\to\infty}\mathfrak{F}_\alpha g_{k}=F'\ \mbox{in}\  L^{2}(\mathbb{R}^n).$$
 Using theorem \ref{P1theo1.1} and the fact that the sequences $\{f_{k}-g_{k}\}$ and $\{\mathfrak{F}_{\alpha}f_{k}-\mathfrak{F}_\alpha g_{k}\}$ converge to $0$ and $F-F',$ respectively in $L^2(\mathbb{R}^n),$ we obtain
 $${\|F-F'\|}_{L^{2}(\mathbb{R}^n)}=0.$$
 Thus, $F=F'.$\\
 Hence, we arrive at the following definition.
 \begin{defn}\label{P1defn1.1}
 If $f\in L^{2}(\mathbb{R}^n)$, the unique  $F\in L^{2}(\mathbb{R}^n)$ is defined as the FrFT of $f$ i.e., $\mathfrak{F_\alpha}f=F$.\\
 For convenience we write, $$(\mathfrak{F}_{\alpha}f)(\boldsymbol\xi)=F(\boldsymbol\xi)=\displaystyle\int_{\mathbb{R}^n}f(\boldsymbol t)K_{\alpha}(\boldsymbol t,\boldsymbol\xi)d\boldsymbol t.$$
 \end{defn}
 \begin{theorem}\label{P1theo1.2}
 The map $\mathfrak{F}_\alpha: L^2(\mathbb{R}^n)\rightarrow L^2(\mathbb{R}^n)$ is a bijection satisfying the Parseval's identity, i.e.,
 $$\|f\|_{L^{2}(\mathbb{R}^n)}=\|\mathfrak{F_\alpha}f\|_{L^{2}(\mathbb{R}^n)},\   \mbox{for all}\ \ f\in L^{2}(\mathbb{R}^n).$$ 
 If $g\in L^{2}(\mathbb{R}^n)$, the unique $f\in L^{2}(\mathbb{R}^n)$ satisfying $\mathfrak{F}_\alpha f=g$ is given by $f=\mathfrak{F}_{-\alpha} g$.
 \end{theorem}
 \begin{proof}
 Let $f\in L^{2}(\mathbb{R}^n)$ and $\{f_{k}\}$ be a sequence in $\mathcal{S}(\mathbb{R}^n)$ converging to $f$ in $L^{2}(\mathbb{R}^n)$. Then, by definition
 $$\mathfrak{F}_\alpha f=\lim_{k\to\infty}\mathfrak{F}_\alpha f_{k}\ \mbox{in} \ L^{2}(\mathbb{R}^n).$$ Also, we have $\|f_{k}\|_{L^{2}(\mathbb{R}^n)}=\|\mathfrak{F_\alpha}f_{k}\|_{L^{2}(\mathbb{R}^n)}.$  
 From which it follows that $$\|f\|_{L^{2}(\mathbb{R}^n)}=\|\mathfrak{F_\alpha}f\|_{L^{2}(\mathbb{R}^n)}.$$\\
 Let $\{\phi_{k}\} $ be a sequence in $\mathcal{S}(\mathbb{R}^n)$ converging to $g$ in $L^{2}(\mathbb{R}^n)$. Using Parseval's identity, we obtain 
 $$\mathfrak{F}_{-\alpha}\phi_{k}\rightarrow\mathfrak{F}_{-\alpha}g\ \mbox{in}\ L^{2}(\mathbb{R}^n).$$ By proposition \ref{P1prop3}, $\mathfrak{F}_{-\alpha}\phi_{k}\in \mathcal{S}(\mathbb{R}^n)$ for all $k$ and hence by the definition \ref{P1defn1.1} 
 $$\mathfrak{F}_\alpha(\mathfrak{F}_{-\alpha}\phi_{k})\rightarrow \mathfrak{F}_\alpha(\mathfrak{F}_{-\alpha}g)\ \mbox{in}\ L^{2}(\mathbb{R}^n).$$
 Thus, it follows that $\mathfrak{F}_\alpha(\mathfrak{F}_{-\alpha}g)=g\ \mbox{in}\ L^{2}(\mathbb{R}^n)$. This completes the proof.
\end{proof}
We mention below some properties of FrFT:
\begin{description}
\item[FrFTP1] (Parseval's formula) If $f,g$ is in $L^2(\mathbb{R}^n),$ then $$\langle f,g\rangle_{L^2(\mathbb{R}^n)}=\langle\mathfrak{F}_{\alpha}f,\mathfrak{F}_{\alpha}g\rangle_{L^2(\mathbb{R}^n)}.$$
\item[FrFTP2] $\mathfrak{F}_{\alpha}(\mathfrak{F}_{\beta}f)=\mathfrak{F}_{\alpha+\beta}f$ for all $f\in L^2(\mathbb{R}^n).$
\end{description}
\subsection{Fractional convolution and their FrFT}
\begin{defn}
Let $f$ and $g$ be two complex valued measurable functions defined on $\mathbb{R}^n$. Then the fractional convolution of $f$ and $g$, denoted by $f\star_\alpha g$, is defined by\\
$$(f\star_\alpha g)(\boldsymbol t)=\int_{\mathbb{R}^n}e^{-\frac{i}{2}(\|\boldsymbol t\|^2-\|\boldsymbol y\|^2)\cot\alpha}f(\boldsymbol y)g(\boldsymbol t-\boldsymbol y)d\boldsymbol y,\ \boldsymbol t\in\mathbb{R}^n,$$
provided the integral on the right is defined.
\end{defn}
Note that the fractional convolution reduces to the traditional convolution if $\alpha=\frac{\pi}{2}.$
\begin{theorem}\label{P1theo1.4}
If $f,g\in L^1(\mathbb{R}^n)$, then $$(\mathfrak{F}_\alpha\{f\star_{\alpha}g\})(\boldsymbol\xi)=\frac{1}{C_\alpha}e^{-\frac{i}{2}\|\boldsymbol\xi\|^2\cot\alpha}(\mathfrak{F}_\alpha f)(\boldsymbol\xi)\left(\mathfrak{F}_\alpha\left\{e^{-\frac{i}{2}\|\cdot\|^2\cot\alpha}g(\cdot)\right\}\right)(\boldsymbol\xi),\ \boldsymbol\xi\in\mathbb{R}^n.$$
\end{theorem}
\begin{proof}
For $\boldsymbol\xi\in\mathbb{R}^n,$
\begin{eqnarray*}
(\mathfrak{F}_\alpha\{f\star_{\alpha}g\})(\boldsymbol\xi)&=&\int_{\mathbb{R}^n}(f\star_{\alpha}g)(\boldsymbol t)K_{\alpha}(\boldsymbol t,\boldsymbol\xi)d\boldsymbol t\\
&=&\int_{\mathbb{R}^n}\left(\int_{\mathbb{R}^n}e^{-\frac{i}{2}(\|\boldsymbol t\|^2-\|\boldsymbol y\|^2)\cot\alpha}f(\boldsymbol y)g(\boldsymbol t-\boldsymbol y)d\boldsymbol y\right)C_{\alpha}e^{\frac{i}{2}(\|\boldsymbol t\|^2+\|\boldsymbol\xi\|^2)\cot\alpha-i\langle\boldsymbol t,\boldsymbol\xi\rangle\csc\alpha}d\boldsymbol t.
\end{eqnarray*}
Substituting $\boldsymbol t-\boldsymbol y=\boldsymbol u,$ we obtain
\begin{eqnarray*}
(\mathfrak{F}_\alpha\{f\star_{\alpha}g\})(\boldsymbol\xi)&=&\int_{\mathbb{R}^n}e^{-i\langle\boldsymbol u,\boldsymbol\xi\rangle\csc\alpha}g(\boldsymbol u)\left(\int_{\mathbb{R}^n}f(\boldsymbol y)C_{\alpha}e^{\frac{i}{2}(\|\boldsymbol y\|^2+\|\boldsymbol\xi\|^2)\cot\alpha-i\langle\boldsymbol y,\boldsymbol\xi\rangle\csc\alpha}d\boldsymbol y\right)d\boldsymbol u\\
&=&\frac{1}{C_{\alpha}}\int_{\mathbb{R}^n}e^{-\frac{i}{2}(\|\boldsymbol u\|^2+\|\boldsymbol\xi\|^2)\cot\alpha}g(\boldsymbol u)K_{\alpha}(\boldsymbol u,\boldsymbol\xi)\left(\int_{\mathbb{R}^n}f(\boldsymbol y)K_\alpha(\boldsymbol y,\boldsymbol\xi)d\boldsymbol y\right)d\boldsymbol u\\
&=&\frac{1}{C_\alpha}e^{-\frac{i}{2}\|\boldsymbol\xi\|^2\cot\alpha}(\mathfrak{F}_\alpha f)(\boldsymbol\xi)\left(\mathfrak{F}_\alpha\left\{e^{-\frac{i}{2}\|\cdot\|^2\cot\alpha}g(\cdot)\right\}\right)(\boldsymbol\xi).
\end{eqnarray*}
This completes the proof.
\end{proof}
\begin{corollary}\label{P1corollary1.1}
For $f,g\in L^1(\mathbb{R}^n)$ and $\textbf{\em a}\in\mathbb{R}_{0}^n,$ $$\ \left(\mathfrak{F}_\alpha\left\{f(\cdot)\star_{\alpha}g\left(\frac{.}{-{\textbf{\em a}}}\right)\right\}\right)(\boldsymbol\xi)=\frac{|\textbf{\em a}|_{p}}{C_\alpha}e^{-\frac{i}{2}\|\textbf{\em a}\boldsymbol\xi\|^2\cot\alpha}(\mathfrak{F}_\alpha f)(\boldsymbol\xi)\left(\mathfrak{F}_\alpha\left\{e^{-\frac{i}{2}\|\cdot\|^2\cot\alpha}g(\cdot)\right\}\right)(-\textbf{\em a}\boldsymbol\xi),~\boldsymbol\xi\in\mathbb{R}^n.$$
\end{corollary}
\begin{corollary}\label{P1corollary1.2}
If $f,g\in L^1(\mathbb{R}^n)$ and $\textbf{\em a}\in\mathbb{R}_{0}^n$, then
$$\left(\mathfrak{F}_{-\alpha}\left\{f(\cdot)\star_{-\alpha}g\left(\frac{.}{-{\textbf{\em a}}}\right)\right\}\right)(\boldsymbol\xi)=\frac{|\textbf{\em a}|_{p}}{C_{\alpha}}e^{-\frac{i}{2}\|\textbf{\em a}\boldsymbol\xi\|^2\cot\alpha}(\mathfrak{F}_{-\alpha}f)(\boldsymbol\xi)\left(\mathfrak{F}_\alpha\left\{e^{-\frac{i}{2}\|\cdot\|^2\cot\alpha}g(\cdot)\right\}\right)(\textbf{\em a}\boldsymbol\xi),\ \mbox{for}\  \boldsymbol\xi\in\mathbb{R}^n.$$
\end{corollary}
\begin{proof}
From the corollary \ref{P1corollary1.1}, we have for $\boldsymbol\xi\in\mathbb{R}^n,$
\begin{eqnarray}\label{P1eqn3}
\left(\mathfrak{F}_{-\alpha}\left\{f(\cdot)\star_{-\alpha}g\left(\frac{.}{-{\textbf{\em a}}}\right)\right\}\right)(\boldsymbol\xi)=\frac{|\boldsymbol a|_{p}}{C_{-\alpha}}e^{\frac{i}{2}\|\boldsymbol a\boldsymbol\xi\|^2\cot\alpha}(\mathfrak{F}_{-\alpha}f)(\boldsymbol\xi)\left(\mathfrak{F}_{-\alpha}\left\{e^{\frac{i}{2}\|\cdot\|^2\cot\alpha}g(\cdot)\right\}\right)(-\textbf{\em a}\boldsymbol\xi),\ \boldsymbol\xi\in\mathbb{R}^n.
\end{eqnarray}
Now,
\begin{eqnarray}\label{P1eqn4}
\left(\mathfrak{F}_{-\alpha}\left\{e^{\frac{i}{2}\|\cdot\|^2\cot\alpha}g(\cdot)\right\}\right)(\boldsymbol\xi)&=&\int_{\mathbb{R}^n}K_{-\alpha}(\boldsymbol t,\boldsymbol\xi)e^{\frac{i}{2}\|\boldsymbol t\|^2\cot\alpha}g(\boldsymbol t)d\boldsymbol t\notag\\
&=&C_{-\alpha}\int_{\mathbb{R}^n}e^{-\frac{i}{2}\|\boldsymbol\xi\|^2\cot\alpha-i\langle\boldsymbol t,-\boldsymbol\xi\rangle\csc\alpha}g(\boldsymbol t)d\boldsymbol t\notag\\
&=&\frac{C_{-\alpha}}{C_{\alpha}}e^{-i\|\boldsymbol\xi\|^2\cot\alpha}\int_{\mathbb{R}^n}e^{-\frac{i}{2}\|\boldsymbol t\|^2\cot\alpha}g(\boldsymbol t)K_{\alpha}(\boldsymbol t,-\boldsymbol\xi)d\boldsymbol t\notag\\
&=&\frac{C_{-\alpha}}{C_{\alpha}}e^{-i\|\boldsymbol\xi\|^2\cot\alpha}\left(\mathfrak{F}_\alpha\left\{e^{-\frac{i}{2}\|\cdot\|^2\cot\alpha}g(\cdot)\right\}\right)(-\boldsymbol\xi). 
\end{eqnarray}
Using equation (\ref{P1eqn4}) in equation (\ref{P1eqn3}), we get the required result.
\end{proof}
\begin{note}
This theorem also holds if $f\in L^1(\mathbb{R}^n),g\in L^2(\mathbb{R}^n)$ and $\boldsymbol a\in\mathbb{R}_{0}^n.$
\end{note}
\subsection{Wavelet admissible functions and their characterization}
\begin{defn}
A non-zero function $\psi\in L^2(\mathbb{R}^n)$ is said to be wavelet admissible if $$\langle\psi_{\textbf{\em a},\textbf{\em b},\alpha },\psi\rangle_{L^2(\mathbb{R}^n)}\in L^2\left(\mathbb{R}^n\times\mathbb{R}_{0}^n,\frac{d\boldsymbol b d\boldsymbol a}{|{\boldsymbol a}|_{p}^2}\right),~\mbox{i.e.},\int_{\mathbb{R}_0^n}\int_{\mathbb{R}^n}\left|\langle\psi_{\textbf{\em a},\textbf{\em b},\alpha },\psi\rangle_{L^2(\mathbb{R}^n)}\right|^2\frac{d\boldsymbol b d\boldsymbol a}{|{\boldsymbol a}|_{p}^2}<\infty,$$
where $\psi_{\boldsymbol a,\boldsymbol b,\alpha}$ is given by equation (\ref{P1eqnC}).
\end{defn}
A wavelet admissible function may also be sometimes called mother wavelet or wavelet in short.\\
For a given wavelet $\psi$ we define a family of fractional wavelets denoted by $\psi_{\textbf{\em a},\boldsymbol b,\alpha}$ and is defined by,\\
\begin{equation}\label{P1eqnC}
\psi_{\textbf{\em a},\boldsymbol b,\alpha}=\frac{1}{\sqrt{|\textbf{\em a}|_{p}}}\psi\left(\frac{\boldsymbol t-\boldsymbol b}{\textbf{\em a}}\right)e^{-\frac{i}{2}({\|\boldsymbol t}\|^2-{\|\boldsymbol b}\|^2)\cot\alpha},
\end{equation}
where $ \textbf{\em a}\in\mathbb{R}_{0}^n,\boldsymbol b\in\mathbb{R}^n.$\\
We have, 
\begin{eqnarray*}
\langle\psi_{\textbf{\em a},\textbf{\em b},\alpha },\psi\rangle_{L^2(\mathbb{R}^n)} &=&
\int_{\mathbb{R}^n}\psi_{\textbf{\em a},\textbf{\em b},\alpha}(\boldsymbol t)\overline{\psi(\boldsymbol t)}d\boldsymbol t\\
&=& \int_{\mathbb{R}^n}\frac{1}{\sqrt{|\textbf{\em a}|_{p}}}\psi\left(\frac{\boldsymbol t-\boldsymbol b}{\textbf{\em a}}\right)e^{-\frac{i}{2}(\|\boldsymbol t\|^2-\|\boldsymbol{b}\|^2)\cot\alpha}\overline{\psi(\boldsymbol t)}d\boldsymbol t\\
&=& \int_{\mathbb{R}^n}\frac{1}{\sqrt{|\textbf{\em a}|_{p}}}\psi\left(\frac{\boldsymbol b-\boldsymbol t}{-\textbf{\em a}}\right)e^{-\frac{i}{2}(\|\boldsymbol b\|^2-\|\boldsymbol{t}\|^2)\cot({-\alpha})}\overline{\psi(\boldsymbol t)}d\boldsymbol t\\
&=& \left(\frac{1}{\sqrt{|\textbf{\em a}|_{p}}}\psi\left(\frac{.}{-\textbf{\em a}}\right)\star_{-\alpha}\overline{\psi(\cdot)}\right)(\boldsymbol b).
\end{eqnarray*}
We reevaluate the admissibility condition:
\begin{equation}\label{P1eqnB}
\hspace{-1cm}\displaystyle\int_{\mathbb{R}_0^n}\int_{\mathbb{R}^n}\left|\langle\psi_{\textbf{\em a},\textbf{\em b},\alpha },\psi\rangle_{L^2(\mathbb{R}^n)}\right|^2\frac{d\boldsymbol b d\boldsymbol a}{|{\boldsymbol a}|_{p}^2}\\[6pt]
=\displaystyle\int_{\mathbb{R}_0^n}\int_{\mathbb{R}^n}\left|\left(\frac{1}{\sqrt{|\textbf{\em a}|_{p}}}\psi\left(\frac{.}{-\textbf{\em a}}\right)\star_{-\alpha}\overline{\psi(\cdot)}\right)(\boldsymbol b)\right|^2\frac{d\boldsymbol b d\boldsymbol a}{|{\boldsymbol a}|_{p}^2}.\\[6pt]
\end{equation}
Using corollary \ref{P1corollary1.2} in equation (\ref{P1eqnB}), leads to
\begin{equation*}
\begin{array}{l}
\hspace{-2cm}\displaystyle\int_{\mathbb{R}_0^n}\int_{\mathbb{R}^n}\left|\langle\psi_{\textbf{\em a},\textbf{\em b},\alpha },\psi\rangle_{L^2(\mathbb{R}^n)}\right|^2\frac{d\boldsymbol b d\boldsymbol a}{|{\boldsymbol a}|_{p}^2}\\
=\displaystyle\int_{\mathbb{R}_0^n}\int_{\mathbb{R}^n}\left|\frac{1}{\sqrt{|\textbf{\em a}|_{p}}}\frac{|\textbf{\em a}|_{p}}{C_{\alpha}}e^{-\frac{i}{2}\|\textbf{\em a}\boldsymbol\xi\|^2\cot\alpha}(\mathfrak{F}_{-\alpha}\bar{\psi})(\boldsymbol\xi)\left(\mathfrak{F}_\alpha\left\{e^{-\frac{i}{2}\|\cdot\|^2\cot\alpha}\psi(\cdot)\right\}\right)(\textbf{\em a}\boldsymbol\xi)\right|^2d\boldsymbol\xi\frac{d\textbf{\em a}}{|{\boldsymbol a}|_{p}^2}\\
=\frac{1}{|C_{\alpha}|^2}\displaystyle\int_{\mathbb{R}_{0}^n}\int_{\mathbb{R}^n}|(\mathfrak{F}_{-\alpha}\bar{\psi})(\boldsymbol\xi)|^{2}\left|\left(\mathfrak{F}_\alpha\left\{e^{-\frac{i}{2}\|\cdot\|^2\cot\alpha}\psi(\cdot)\right\}\right)(\textbf{\em a}\boldsymbol\xi)\right|^2\frac{d\boldsymbol\xi d\textbf{\em a}}{|\textbf{\em a}|_{p}}\\[6pt]
= \frac{1}{|C_{\alpha}|^2}\displaystyle\int_{\mathbb{R}^n}|(\mathfrak{F}_{-\alpha}\bar{\psi})(\boldsymbol\xi)|^{2}\left(\int_{\mathbb{R}_{0}^n}\left|\left(\mathfrak{F}_\alpha\left\{e^{-\frac{i}{2}\|\cdot\|^2\cot\alpha}\psi(\cdot)\right\}\right)(\textbf{\em a}\boldsymbol\xi)\right|^2\frac{d\textbf{\em a}}{|\textbf{\em a}|_{p}}\right)d\boldsymbol\xi \\[6pt]
=\frac{1}{|C_{\alpha}|^2}\displaystyle\int_{\mathbb{R}^n}|(\mathfrak{F}_{-\alpha}\bar{\psi})(\boldsymbol\xi)|^{2}\left(\int_{\mathbb{R}_{0}^n}\left|\left(\mathfrak{F}_\alpha\left\{e^{-\frac{i}{2}\|\cdot\|^2\cot\alpha}\psi(\cdot)\right\}\right)(\boldsymbol u)\right|^2\frac{1}{|\boldsymbol u|_{p}} d\boldsymbol u\right)d\boldsymbol\xi \\[10pt]
=\frac{1}{|C_{\alpha}|^2}C_{\psi ,\alpha}\|\psi\|^2_{L^2(\mathbb{R}^n)},
\end{array}
\end{equation*}
\begin{equation}\label{P1eqn5}
\hspace{-10cm}\mbox{where}\ C_{\psi ,\alpha}=\displaystyle\int_{\mathbb{R}_{0}^n}\left|\left(\mathfrak{F}_\alpha\left\{e^{-\frac{i}{2}\|\cdot\|^2\cot\alpha}\psi(\cdot)\right\}\right)(\boldsymbol u)\right|^2\frac{1}{|\boldsymbol u|_{p}} d\boldsymbol u.
\end{equation}
As a consequence of which we have the following characterization of the wavelet admissible function.
\begin{theorem}\label{P1theo1.6}
A non-zero function $\psi\in L^2(\mathbb{R}^n)$ is wavelet admissible if and only if $C_{\psi,\alpha},$ as defined in (\ref{P1eqn5}), is finite, i.e.,
$$\int_{\mathbb{R}_{0}^n}\left|\left(\mathfrak{F}_\alpha\left\{e^{-\frac{i}{2}\|\cdot\|^2\cot\alpha}\psi(\cdot)\right\}\right)(\boldsymbol u)\right|^2\frac{1}{|\boldsymbol u|_{p}} d\boldsymbol u<\infty.$$
\end{theorem}
\section{Continuous Fractional Wavelet Transform(CFrWT)}
\begin{defn}
Let $f$ be a complex valued measurable function on $\mathbb{R}^n.$ Then the CFrWT of $f$  with respect to the wavelet  $\psi$, denoted by $W_{\psi,\alpha}f$, is defined by
$$(W_{\psi,\alpha}f)(\textbf{\em b},\boldsymbol a)=\int_{\mathbb{R}^n}f(\boldsymbol t)\overline{\psi_{\textbf{\em a},\boldsymbol b,\alpha}(\boldsymbol t)}d\boldsymbol t, (\textbf{\em b},\boldsymbol a)\in \mathbb{R}^n\times\mathbb{R}_{0}^n,$$
provided the integral on the right exists.\\
Note that the CFrWT reduces to the classical wavelet transform in $n$-dimensions if $\alpha=\frac{\pi}{2},$ i.e., $W_{\psi,\frac{\pi}{2}}=W_{\psi}.$\\
\end{defn}
\begin{theorem}\label{P1theo2.1}
Let $\psi$ be a wavelet and $f\in L^2(\mathbb{R}^n)$, then $W_{\psi,\alpha}f\in L^2\left(\mathbb{R}^n\times\mathbb{R}_{0}^n,\frac{d\boldsymbol b d\textbf{\em a}}{|\textbf{\em a}|^2_{p}}\right).$
\end{theorem}
\begin{proof}
Clearly, for $(\boldsymbol b,\boldsymbol a)\in\mathbb{R}^n\times\mathbb{R}_{0}^n,$ we have
$$(W_{\psi,\alpha}f)(\boldsymbol b,\boldsymbol a)=\langle f,\psi_{\boldsymbol a,\boldsymbol b,\alpha}\rangle_{L^2(\mathbb{R}^n)}=\frac{1}{\sqrt{|\textbf{\em a}|_{p}}}\psi\left(\frac{.}{-\textbf{\em a}}\right)\star_{-\alpha}\overline{f(\cdot)}.$$
Therefore,
\begin{eqnarray*}
\int_{\mathbb{R}_0^n}\int_{\mathbb{R}^n}|(W_{\psi,\alpha}f)(\textbf{\em b},\boldsymbol a)|^2\frac{d\boldsymbol b d\boldsymbol a}{|{\boldsymbol a}|_{p}^2}&=&\int_{\mathbb{R}_0^n}\int_{\mathbb{R}^n}\left|\left(\frac{1}{\sqrt{|\textbf{\em a}|_{p}}}\psi\left(\frac{.}{-\textbf{\em a}}\right)\star_{-\alpha}\overline{f(\cdot)}\right)(\boldsymbol b)\right|^2\frac{d\boldsymbol b d\boldsymbol a}{|{\boldsymbol a}|_{p}^2}\notag\\
&=&\int_{\mathbb{R}_0^n}\int_{\mathbb{R}^n}\left|\frac{1}{\sqrt{|\textbf{\em a}|_{p}}}\frac{|\textbf{\em a}|_{p}}{C_{\alpha}}e^{-\frac{i}{2}\|\textbf{\em a}\boldsymbol\xi\|^2\cot\alpha}(\mathfrak{F}_{-\alpha}\bar{f})(\boldsymbol\xi)\left(\mathfrak{F}_\alpha\left\{e^{-\frac{i}{2}\|\cdot\|^2\cot\alpha}\psi(\cdot)\right\}\right)(\textbf{\em a}\boldsymbol\xi)\right|^2d\boldsymbol\xi\frac{d\textbf{\em a}}{|{\boldsymbol a}|_{p}^2}\notag\\
&=&\frac{1}{|C_{\alpha}|^2}\int_{\mathbb{R}_{0}^n}\int_{\mathbb{R}^n}|(\mathfrak{F}_{-\alpha}\bar{f})(\boldsymbol\xi)|^{2}\left|\left(\mathfrak{F}_\alpha\left\{e^{-\frac{i}{2}\|\cdot\|^2\cot\alpha}\psi(\cdot)\right\}\right)(\textbf{\em a}\boldsymbol\xi)\right|^2\frac{d\boldsymbol\xi d\textbf{\em a}}{|\textbf{\em a}|_{p}}.
\end{eqnarray*}
Using Fubini's theorem, we have
\begin{eqnarray}\label{P1eqn6}
\int_{\mathbb{R}_0^n}\int_{\mathbb{R}^n}|(W_{\psi,\alpha}f)(\textbf{\em b},\boldsymbol a)|^2\frac{d\boldsymbol b d\boldsymbol a}{|{\boldsymbol a}|_{p}^2}
&=&\frac{1}{|C_{\alpha}|^2}\int_{\mathbb{R}^n}|(\mathfrak{F}_{-\alpha}\bar{f})(\boldsymbol\xi)|^{2}\left(\int_{\mathbb{R}_{0}^n}\left|\left(\mathfrak{F}_\alpha\left\{e^{-\frac{i}{2}\|\cdot\|^2\cot\alpha}\psi(\cdot)\right\}\right)(\textbf{\em a}\boldsymbol\xi)\right|^2\frac{d\textbf{\em a}}{|\textbf{\em a}|_{p}}\right)d\boldsymbol\xi \notag\\
&=&\frac{1}{|C_{\alpha}|^2}\int_{\mathbb{R}^n}|(\mathfrak{F}_{-\alpha}\bar{f})(\boldsymbol\xi)|^{2}\left(\int_{\mathbb{R}_{0}^n}\left|\left(\mathfrak{F}_\alpha\left\{e^{-\frac{i}{2}\|\cdot\|^2\cot\alpha}\psi(\cdot)\right\}\right)(\boldsymbol u)\right|^2\frac{1}{|\boldsymbol u|_{p}} d\boldsymbol u\right)d\boldsymbol\xi \notag\\
&=&\frac{1}{|C_{\alpha}|^2}C_{\psi ,\alpha}\|f\|^2_{L^2(\mathbb{R}^n)}. 
\end{eqnarray}
Since $\psi$ is a wavelet and $f\in L^2(\mathbb{R}^n),$ from equation (\ref{P1eqn6}) it follows that $\displaystyle\int_{\mathbb{R}_0^n}\int_{\mathbb{R}^n}|(W_{\psi,\alpha}f)(\textbf{\em b},\boldsymbol a)|^2\frac{d\boldsymbol b d\boldsymbol a}{|{\boldsymbol a}|_{p}^2}<\infty.$\\
Hence, the theorem follows.
\end{proof} 
We now prove some lemmas that will be used in proving the inner product relation of CFrWT.
\begin{lemma}\label{P1lemma2.1}
(FrFT of $\psi_{\boldsymbol a,\boldsymbol b,\alpha}$) Let $\psi$ be a wavelet and $f\in L^2(\mathbb{R}^n)$, then
$$(\mathfrak{F}_{\alpha}\psi_{\textbf{\em a},\boldsymbol b,\alpha})(\boldsymbol\xi)=\sqrt{|\textbf{\em a}|_{p}} e^{\frac{i}{2}(\|\boldsymbol b\|^2+\|\boldsymbol\xi\|^2)\cot\alpha-i\langle\boldsymbol b,\boldsymbol\xi\rangle\csc\alpha-\frac{i}{2}\|\textbf{\em a}\boldsymbol\xi\|^2\cot\alpha}\left(\mathfrak{F_{\alpha}}\left\{e^{-\frac{i}{2}\|\cdot\|^2\cot\alpha}\psi(\cdot)\right\}\right)(\textbf{\em a}\boldsymbol\xi).$$
\end{lemma}
\begin{proof}
For $\boldsymbol\xi\in\mathbb{R}^n$, we have
\begin{eqnarray*}
(\mathfrak{F}_{\alpha}\psi_{\textbf{\em a},\boldsymbol b,\alpha})(\boldsymbol\xi)&=&\int_{\mathbb{R}^n}\psi_{\textbf{\em a},\boldsymbol b,\alpha}K_{\alpha}(\boldsymbol t,\boldsymbol\xi)d\boldsymbol t\\
&=&\int_{\mathbb{R}^n}\frac{C_{\alpha}}{\sqrt{|\textbf{\em a}|_{p}}}\psi\left(\frac{\boldsymbol t-\boldsymbol b}{\textbf{\em a}}\right)e^{\frac{i}{2}(\|\boldsymbol b\|^2+\|\boldsymbol\xi\|^2)\cot\alpha-i\langle\boldsymbol t,\boldsymbol\xi\rangle\csc\alpha}d\boldsymbol t.\\
\end{eqnarray*}
Putting $\frac{\boldsymbol t-\boldsymbol b}{\textbf{\em a}}=\boldsymbol u,$ we have
\begin{eqnarray*}
(\mathfrak{F}_{\alpha}\psi_{\textbf{\em a},\boldsymbol b,\alpha})(\boldsymbol\xi)&=&\int_{\mathbb{R}^n}C_{\alpha}\sqrt{|\textbf{\em a}|_{p}}\psi(\boldsymbol u)e^{\frac{i}{2}(\|\boldsymbol b\|^2+\|\boldsymbol\xi\|^2)\cot\alpha-i\langle\boldsymbol b+\textbf{\em a}\boldsymbol u,\boldsymbol\xi\rangle\csc\alpha}d\boldsymbol u\\
&=&\sqrt{|\textbf{\em a}|_{p}}e^{\frac{i}{2}(\|\boldsymbol b\|^2+\|\boldsymbol\xi\|^2)\cot\alpha-i\langle\boldsymbol b,\boldsymbol\xi\rangle\csc\alpha}\int_{\mathbb{R}^n}C_{\alpha}\psi(\boldsymbol u)e^{-i\langle\textbf{\em a}\boldsymbol u,\boldsymbol\xi\rangle\csc\alpha}d\boldsymbol u\\
&=&\sqrt{|\textbf{\em a}|_{p}}e^{\frac{i}{2}(\|\boldsymbol b\|^2+\|\boldsymbol\xi\|^2)\cot\alpha-i\langle\boldsymbol b,\boldsymbol\xi\rangle\csc\alpha-\frac{i}{2}\|\textbf{\em a}\boldsymbol\xi\|^2\cot\alpha}\int_{\mathbb{R}^n}e^{-\frac{i}{2}\|\boldsymbol u\|^2\cot\alpha}\psi(\boldsymbol u)K_{\alpha}(\boldsymbol u,\textbf{\em a}\boldsymbol\xi)d\boldsymbol u\\
&=&\sqrt{|\textbf{\em a}|_{p}}e^{\frac{i}{2}(\|\boldsymbol b\|^2+\|\boldsymbol\xi\|^2)\cot\alpha-i\langle\boldsymbol b,\boldsymbol\xi\rangle\csc\alpha-\frac{i}{2}\|\textbf{\em a}\boldsymbol\xi\|^2\cot\alpha}\left(\mathfrak{F}_{\alpha}\left\{e^{-\frac{i}{2}\|\cdot\|^2\cot\alpha}\psi(\cdot)\right\}\right)(\textbf{\em a}\boldsymbol\xi).\\
\end{eqnarray*}
This proves the lemma.
\end{proof}
\begin{lemma}\label{P1lemma2.2}
\big(FrFT of $\overline{\psi_{\boldsymbol a,\boldsymbol b,\alpha}}$\big) Let $\psi$ be a wavelet and $f\in L^2(\mathbb{R}^n)$, then
$$\left(\mathfrak{F}_{\alpha,\boldsymbol b}\left\{\overline{\psi_{\textbf{\em a},\boldsymbol b,\alpha}(\boldsymbol t)}\right\}\right)(\boldsymbol\xi)=\sqrt{|\textbf{\em a}|_{p}}\frac{C_{\alpha}}{\overline{C_{\alpha}}} e^{\frac{i}{2}(\|\boldsymbol t\|^2+\|\boldsymbol\xi\|^2)\cot\alpha-i\langle\boldsymbol t,\boldsymbol\xi\rangle\csc\alpha+\frac{i}{2}\|\textbf{\em a}\boldsymbol\xi\|^2\cot\alpha}\overline{\left(\mathfrak{F_{\alpha}}\left\{e^{-\frac{i}{2}\|\cdot\|^2\cot\alpha}\psi(\cdot)\right\}\right)(\textbf{\em a}\boldsymbol\xi)},$$
where $\mathfrak{F}_{\alpha,\boldsymbol b}$ denotes the FrFT to be taken with respect  to the argument $\boldsymbol b.$ 
\end{lemma}
\begin{proof}
Using the definition of $\mathfrak{F}_{\alpha,\boldsymbol b}$ and $\psi_{\textbf{\em a},\boldsymbol b,\alpha}$, we have
\begin{eqnarray}\label{P1eqn7}
\left(\mathfrak{F}_{\alpha,\boldsymbol b}\left\{\overline{\psi_{\textbf{\em a},\boldsymbol b,\alpha}(\boldsymbol t)}\right\}\right)(\boldsymbol\xi)&=&\frac{C_{\alpha}}{\sqrt{|\textbf{\em a}|_{p}}}e^{\frac{i}{2}(\|\boldsymbol t\|^2+\|\boldsymbol\xi\|^2)\cot\alpha}\overline{\int_{\mathbb{R}^n}\psi\left(\frac{\boldsymbol t-\boldsymbol b}{\textbf{\em a}}\right)e^{i\langle\boldsymbol b,\boldsymbol\xi\rangle\csc\alpha} d\boldsymbol b}. 
\end{eqnarray}
Putting $\frac{\boldsymbol t-\boldsymbol b}{\textbf{\em a}}=\boldsymbol u$ in equation (\ref{P1eqn7}), we have
\begin{eqnarray*}
\left(\mathfrak{F}_{\alpha,\boldsymbol b}\left\{\overline{\psi_{\textbf{\em a},\boldsymbol b,\alpha}(\boldsymbol t)}\right\}\right)(\boldsymbol\xi)&=&\sqrt{|\textbf{\em a}|_{p}}C_{\alpha} e^{\frac{i}{2}(\|\boldsymbol t\|^2+\|\boldsymbol\xi\|^2)\cot\alpha}\overline{\int_{\mathbb{R}^n}\psi(\boldsymbol u)e^{i\langle\boldsymbol t-\textbf{\em a}\boldsymbol u,\boldsymbol\xi\rangle\csc\alpha} d\boldsymbol u}\\
&=& \sqrt{|\textbf{\em a}|_{p}}C_{\alpha} e^{\frac{i}{2}(\|\boldsymbol t\|^2+\|\boldsymbol\xi\|^2)\cot\alpha-i\langle\boldsymbol t,\boldsymbol\xi\rangle\csc\alpha }\overline{\int_{\mathbb{R}^n}\psi(\boldsymbol u)e^{-i\langle\boldsymbol u,\textbf{\em a}\boldsymbol\xi\rangle\csc\alpha} d\boldsymbol u}\\
&=&\sqrt{|\textbf{\em a}|_{p}}\frac{C_{\alpha}}{\overline{C_{\alpha}}} e^{\frac{i}{2}(\|\boldsymbol t\|^2+\|\boldsymbol\xi\|^2)\cot\alpha-i\langle\boldsymbol t,\boldsymbol\xi\rangle\csc\alpha }\overline{\int_{\mathbb{R}^n}\psi(\boldsymbol u)K_{\alpha}(\boldsymbol t,\textbf{\em a}\boldsymbol\xi)e^{-\frac{i}{2}(\|\boldsymbol u\|^2+\|\textbf{\em a}\boldsymbol\xi\|^2)\cot\alpha} d\boldsymbol u}\\
&=&\sqrt{|\textbf{\em a}|_{p}}\frac{C_{\alpha}}{\overline{C_{\alpha}}} e^{\frac{i}{2}(\|\boldsymbol t\|^2+\|\boldsymbol\xi\|^2)\cot\alpha-i\langle\boldsymbol t,\boldsymbol\xi\rangle\csc\alpha+\frac{i}{2}\|\textbf{\em a}\boldsymbol\xi\|^2\cot\alpha }\overline{\int_{\mathbb{R}^n}\psi(\boldsymbol u)K_{\alpha}(\boldsymbol t,\textbf{\em a}\boldsymbol\xi)e^{-\frac{i}{2}(\|\boldsymbol u\|^2+\|\textbf{\em a}\boldsymbol\xi\|^2)\cot\alpha} d\boldsymbol u}\\
&=&\sqrt{|\textbf{\em a}|_{p}}\frac{C_{\alpha}}{\overline{C_{\alpha}}} e^{\frac{i}{2}(\|\boldsymbol t\|^2+\|\boldsymbol\xi\|^2)\cot\alpha-i\langle\boldsymbol t,\boldsymbol\xi\rangle\csc\alpha+\frac{i}{2}\|\textbf{\em a}\boldsymbol\xi\|^2\cot\alpha }\overline{\left(\mathfrak{F}_{\alpha}\left\{e^{-\frac{i}{2}\|\cdot\|^2\cot\alpha}\psi(\cdot)\right\}\right)(\textbf{\em a}\boldsymbol \xi)}.
\end{eqnarray*}
This completes the proof.
\end{proof} 
\begin{lemma}\label{P1lemma2.3}
(FrFT of $W_{\psi,\alpha}(\cdot,\boldsymbol a)$) Let $\psi$ be a wavelet and $f\in L^2(\mathbb{R}^n)$, then
$$\left(\mathfrak{F}_{\alpha,\boldsymbol b}\left\{(W_{\psi,\alpha}f)(\textbf{\em b},\boldsymbol a)\right\}\right)(\boldsymbol\xi)=\frac{\sqrt{|\textbf{\em a}|_{p}}}{C_{\alpha}}e^{\frac{i}{2}\|\textbf{\em a}\boldsymbol\xi\|^2\cot\alpha }\overline{\left(\mathfrak{F}_{\alpha}\left\{e^{-\frac{i}{2}\|\cdot\|^2\cot\alpha}\psi(\cdot)\right\}\right)(\textbf{\em a}\boldsymbol \xi)}(\mathfrak{F}_{\alpha}f)(\boldsymbol\xi),$$
where $\mathfrak{F}_{\alpha,\boldsymbol b}$ denotes the FrFT to be taken with respect  to the argument $\boldsymbol b.$ 
\end{lemma} 
\begin{proof}
We have
$$(W_{\psi,\alpha}f)(\textbf{\em b},\boldsymbol a)=\langle f,\psi_{\textbf{\em a},\boldsymbol b,\alpha}\rangle_{L^2(\mathbb{R}^n)}.$$
Using Parseval's formula, we have
\begin{eqnarray*}
(W_{\psi,\alpha}f)(\textbf{\em b},\boldsymbol a)&=&\langle\mathfrak{F}_{\alpha}f,\mathfrak{F}_{\alpha}\psi_{\textbf{\em a},\boldsymbol b,\alpha}\rangle_{L^2(\mathbb{R}^n)}\\
&=&
\sqrt{|\textbf{\em a}|_{p}}\int_{\mathbb{R}^n}e^{-\frac{i}{2}(\|\boldsymbol b\|^2+\|\boldsymbol\xi\|^2)\cot\alpha+i\langle\boldsymbol b,\boldsymbol\xi\rangle\csc\alpha+\frac{i}{2}\|\textbf{\em a}\boldsymbol\xi\|^2\cot\alpha }\overline{\left(\mathfrak{F}_{\alpha}\left\{e^{-\frac{i}{2}\|.\|^2\cot\alpha}\psi(.)\right\}\right)(\textbf{\em a}\boldsymbol \xi)}(\mathfrak{F}_{\alpha}f)(\boldsymbol\xi)d\boldsymbol\xi.
\end{eqnarray*}
Let 
\begin{equation}\label{P1eqn8}
F(\boldsymbol\xi)=\overline{\left(\mathfrak{F}_{\alpha}\left\{e^{-\frac{i}{2}\|\cdot\|^2\cot\alpha}\psi(\cdot)\right\}\right)(\textbf{\em a}\boldsymbol \xi)}(\mathfrak{F}_{\alpha}f)(\boldsymbol\xi).
\end{equation}
Then, we have
\begin{eqnarray}\label{P1eqn9}
(W_{\psi,\alpha}f)(\textbf{\em b},\boldsymbol a)&=&\frac{\sqrt{|\textbf{\em a}|_{p}}}{\overline{C_{\alpha}}}\int_{\mathbb{R}^n}K_{-\alpha}(\boldsymbol\xi,\boldsymbol b)e^{\frac{i}{2}\|\textbf{\em a}\boldsymbol\xi\|^2\cot\alpha}F(\boldsymbol\xi)d\boldsymbol\xi\notag\\
&=&\frac{\sqrt{|\textbf{\em a}|_{p}}}{\overline{C_{\alpha}}}\left(\mathfrak{F}_{-\alpha}\left\{e^{\frac{i}{2}\|\textbf{\em a}\boldsymbol\xi\|^2\cot\alpha}F(\boldsymbol\xi)\right\}\right)(\boldsymbol b).
\end{eqnarray}
From equation (\ref{P1eqn9}), we get
\begin{equation}\label{P1eqn10}
(\mathfrak{F}_{\alpha,\boldsymbol b}\{(W_{\psi,\alpha}f)(\textbf{\em b},\boldsymbol a)\})(\boldsymbol\xi)=\frac{\sqrt{|\textbf{\em a}|_{p}}}{\overline{C_{\alpha}}}e^{\frac{i}{2}\|\textbf{\em a}\boldsymbol\xi\|^2\cot\alpha}F(\boldsymbol\xi).
\end{equation}
Using equation (\ref{P1eqn8}) in equation (\ref{P1eqn10}), we get the required result.
\end{proof}
\subsection{Inner product relation for CFrWT involving two wavelets}
\begin{theorem}\label{P1theo2.2}
Let $\phi,\psi$ be two wavelets satisfying 
\begin{equation}\label{P1eqn13}
\int_{\mathbb{R}_{0}^n}\left|\left(\mathfrak{F}_\alpha\left\{e^{-\frac{i}{2}\|\cdot\|^2\cot\alpha}\phi(\cdot)\right\}\right)(\boldsymbol u)\right|\left|\left(\mathfrak{F}_\alpha\left\{e^{-\frac{i}{2}\|\cdot\|^2\cot\alpha}\psi(\cdot)\right\}\right)(\boldsymbol u)\right|\frac{1}{|\boldsymbol u|_{p}}d\boldsymbol u<\infty
\end{equation}
and $W_{\phi,\alpha}f ,\  W_{\psi,\alpha}g$ denote the CFrWT of $f$ and $g$ with respect to $\phi$ and $\psi$ respectively, then
$$\int_{\mathbb{R}_{0}^n}\int_{\mathbb{R}^n}(W_{\phi,\alpha}f)(\textbf{\em b},\boldsymbol a)\overline{(W_{\psi,\alpha}g)(\textbf{\em b},\boldsymbol a)}\frac{d\boldsymbol b d\textbf{\em a}}{|\textbf{\em a}|_{p}^2}=\frac{C_{\boldsymbol\phi,\boldsymbol\psi,\alpha}}{|C_{\alpha}|^2}\langle f,g\rangle_{L^2(\mathbb{R}^n)},$$
where 
\begin{equation}\label{P1eqn14}
C_{\phi,\psi,\alpha}=\int_{\mathbb{R}_{0}^n}\overline{\left(\mathfrak{F}_\alpha\left\{e^{-\frac{i}{2}\|\cdot\|^2\cot\alpha}\phi(\cdot)\right\}\right)(\boldsymbol u)}\left(\mathfrak{F}_\alpha\left\{e^{-\frac{i}{2}\|\cdot\|^2\cot\alpha}\psi(\cdot)\right\}\right)(\boldsymbol u)\frac{1}{|\boldsymbol u|_{p}}d\boldsymbol u.
\end{equation}
\end{theorem}
\begin{proof}
Treating $(W_{\phi,\alpha}f)(\textbf{\em b},\boldsymbol a),\ (W_{\psi,\alpha}g)(\textbf{\em b},\boldsymbol a)$ as  $L^2({R}^n)$ functions of the variable $\boldsymbol b$ and applying the Parseval's formula for the FrFT, we have
\begin{eqnarray*}
\int_{\mathbb{R}_{0}^n}\int_{\mathbb{R}^n}(W_{\phi,\alpha}f)(\textbf{\em b},\boldsymbol a)\overline{(W_{\psi,\alpha}g)(\textbf{\em b},\boldsymbol a)}\frac{d\boldsymbol b d\textbf{\em a}}{|\textbf{\em a}|_{p}^2}&=&\int_{\mathbb{R}_{0}^n}\left(\int_{\mathbb{R}^n}(\mathfrak{F_{{\alpha},\boldsymbol b}}\{(W_{\phi,\alpha}f)(\textbf{\em b},\boldsymbol a)\})(\boldsymbol\xi)\overline{(\mathfrak{F_{{\alpha},\boldsymbol b}}\{(W_{\psi,\alpha}g)(\textbf{\em b},\boldsymbol a)\})(\boldsymbol\xi)}d\boldsymbol\xi\right)\frac{ d\textbf{\em a}}{|\textbf{\em a}|_{p}^2},\\
\end{eqnarray*}
Using lemma (\ref{P1lemma2.3}), we obtain\\\\
$\displaystyle\int_{\mathbb{R}_{0}^n}\int_{\mathbb{R}^n}(W_{\phi,\alpha}f)(\textbf{\em b},\boldsymbol a)\overline{(W_{\psi,\alpha}g)(\textbf{\em b},\boldsymbol a)}\frac{d\boldsymbol b d\textbf{\em a}}{|\textbf{\em a}|_{p}^2}$
\begin{eqnarray}\label{P1eqn15}
&=&\int_{\mathbb{R}_{0}^n}\left(\int_{\mathbb{R}^n}\frac{|\textbf{\em a}|_{p}}{|C_{\alpha}|^2}\overline{\left(\mathfrak{F_{\alpha}}\left\{e^{-\frac{i}{2}\|\cdot\|^2\cot\alpha}\phi(\cdot)\right\}\right)(\textbf{\em a}\boldsymbol\xi)}(\mathfrak{F_\alpha}f)(\boldsymbol\xi)\left(\mathfrak{F_{\alpha}}\left\{e^{-\frac{i}{2}\|\cdot\|^2\cot\alpha}\psi(\cdot)\right\}\right)(\textbf{\em a}\boldsymbol\xi)\overline{(\mathfrak{F_{\alpha}g})(\boldsymbol\xi))}d\boldsymbol\xi\right)\frac{d\textbf{\em a}}{|\textbf{\em a}|_{p}^2}\notag\\
&=&\int_{\mathbb{R}^n}\frac{1}{|C_{\alpha}|^2}(\mathfrak{F_{\alpha}f})(\boldsymbol\xi)\overline{(\mathfrak{F_{\alpha}g})(\boldsymbol\xi)}\left(\int_{\mathbb{R}_{0}^n}\overline{\left(\mathfrak{F_{\alpha}}\left\{e^{-\frac{i}{2}\|\cdot\|^2\cot\alpha}\phi(\cdot)\right\}\right)(\textbf{\em a}\boldsymbol\xi)}\left(\mathfrak{F_{\alpha}}\left\{e^{-\frac{i}{2}\|\cdot\|^2\cot\alpha}\psi(\cdot)\right\}\right)(\textbf{\em a}\boldsymbol\xi)\frac{d\textbf{\em a}}{|\textbf{\em a}| _{p}}\right)d\boldsymbol\xi. 
\end{eqnarray}
For the integral in the parentheses of (\ref{P1eqn15}) we substitute $\textbf{\em a}\boldsymbol\xi=\boldsymbol u,$ and obtain
\begin{equation}
\begin{array}{l}\label{P1eqn16}
\hspace{-3cm}\displaystyle\int_{\mathbb{R}_{0}^n}\overline{\left(\mathfrak{F_{\alpha}}\left\{e^{-\frac{i}{2}\|\cdot\|^2\cot\alpha}\phi(\cdot)\right\}\right)(\textbf{\em a}\boldsymbol\xi)}\left(\mathfrak{F_{\alpha}}\left\{e^{-\frac{i}{2}\|\cdot\|^2\cot\alpha}\psi(\cdot)\right\}\right)(\textbf{\em a}\boldsymbol\xi)\frac{d\textbf{\em a}}{|\textbf{\em a}| _{p}}\\
\hspace{1cm}=\displaystyle\int_{\mathbb{R}_{0}^n}\left|\left(\mathfrak{F}_\alpha\left\{e^{-\frac{i}{2}\|\cdot\|^2\cot\alpha}\phi(\cdot)\right\}\right)(\boldsymbol u)\right|\left|\left(\mathfrak{F}_\alpha\left\{e^{-\frac{i}{2}\|\cdot\|^2\cot\alpha}\psi(\cdot)\right\}\right)(\boldsymbol u)\right|\frac{1}{|\boldsymbol u|_{p}}d\boldsymbol u.
\end{array}
\end{equation}
From (\ref{P1eqn14}), (\ref{P1eqn15})and (\ref{P1eqn16}), we have
\begin{eqnarray*}
\int_{\mathbb{R}_{0}^n}\int_{\mathbb{R}^n}(W_{\phi,\alpha}f)(\textbf{\em b},\boldsymbol a)\overline{(W_{\psi,\alpha}g)(\textbf{\em b},\boldsymbol a)}\frac{d\boldsymbol b d\textbf{\em a}}{|\textbf{\em a}|_{p}^2}&=&\frac{1}{|C_{\alpha}|^2}\int_{\mathbb{R}^n}C_{\boldsymbol\phi,\boldsymbol\psi,\alpha}(\mathfrak{F_{\alpha}f})(\boldsymbol\xi)\overline{(\mathfrak{F_{\alpha}g})(\boldsymbol\xi)}d\boldsymbol\xi\\
&=&\frac{C_{\boldsymbol\phi,\boldsymbol\psi,\alpha}}{|C_{\alpha}|^2}\langle\frak{F_{\alpha}f},\mathfrak{F_{\alpha}g}\rangle_{L^2(\mathbb{R}^n)}\\
&=&\frac{C_{\boldsymbol\phi,\boldsymbol\psi,\alpha}}{|C_{\alpha}|^2}\langle f,g\rangle_{L^2(\mathbb{R}^n)},\ \mbox{using the Parseval's formula}.
\end{eqnarray*}
This completes the proof.
\end{proof}
\begin{corollary}\label{P1corollary2.1}
Let $\psi$ is a wavelet and $f,g\in L^2(\mathbb{R}^n)$. Then
$$\int_{\mathbb{R}_{0}^n}\int_{\mathbb{R}^n}(W_{\psi,\alpha}f)(\textbf{\em b},\boldsymbol a)\overline{(W_{\psi,\alpha}g)(\textbf{\em b},\boldsymbol a)}\frac{d\boldsymbol b d\textbf{\em a}}{|\textbf{\em a}|_{p}^2}=\frac{C_{\boldsymbol\psi,\alpha}}{|C_{\alpha}|^2}\langle f,g\rangle_{L^2(\mathbb{R}^n)},$$
where $C_{\psi ,\alpha}$ is as defined in (\ref{P1eqn5}).
\end{corollary}
\begin{corollary}\label{P1corollary2.2}
Let $\psi$ is a wavelet and $f\in L^2(\mathbb{R}^n)$. Then
$$\int_{\mathbb{R}_{0}^n}\int_{\mathbb{R}^n}|(W_{\psi,\alpha}f)(\textbf{\em b},\boldsymbol a)|^2\frac{d\boldsymbol b d\textbf{\em a}}{|\textbf{\em a}|_{p}^2}=\frac{C_{\boldsymbol\psi,\alpha}}{|C_{\alpha}|^2}\|f\|^2_{L^2(\mathbb{R}^n)},$$
where $C_{\psi ,\alpha}$ is as defined in (\ref{P1eqn5}).
\end{corollary} 
From corollary (\ref{P1corollary2.2}) it follows that the linear operator $W_{\psi,\alpha}$ is isometric upto the constant $\frac{\sqrt{C_{\psi,\alpha}}}{|C_{\alpha}|},$ i.e., \\
$$\|W_{\psi,\alpha}f\|_{L^2\Big(\mathbb{R}^n\times\mathbb{R}_{0}^n,\frac{d\boldsymbol b d\textbf{\em a}}{|\textbf{\em a}|_{p}^2}\Big)}=\frac{\sqrt{C_{\psi,\alpha}}}{|C_{\alpha}|}\|f\|_{L^2(\mathbb{R}^n)}\ \mbox{for all}\  f\in L^2(\mathbb{R}^n).$$
\subsection{Reconstruction formula involving two wavelets}
\begin{theorem}\label{P1theo2.3}
Let $\phi,\psi$ be two wavelets satisfying (\ref{P1eqn13})
and $f\in L^2(\mathbb{R}^n)$. Also let $C_{\phi,\psi,\alpha},$ as defined in (\ref{P1eqn14}), is non-zero. Then the reconstruction formula for $f$ is given by
$$f(\boldsymbol t)=\frac{|C_{\alpha}|^2}{C_{\phi,\psi,\alpha}}\int_{\mathbb{R}_{0}^n}\int_{\mathbb{R}^n}\phi_{\textbf{\em a},\boldsymbol b,\alpha}(\boldsymbol t)(W_{\psi,\alpha}f)(\textbf{\em b},\boldsymbol a)\frac{d\boldsymbol b d\textbf{\em a}}{|\textbf{\em a}|_{p}^2}.$$
\end{theorem}
\begin{proof}
Treating $\phi_{\textbf{\em a},\boldsymbol b,\alpha}(\boldsymbol t),(W_{\psi,\alpha}f)(\textbf{\em b},\boldsymbol a)$ as a $L^2(\mathbb{R}^n)$ function of the variable $\boldsymbol b$ and applying the Parseval'  s formula for the FrFT, we have
\begin{align}\label{P1eqn17}
\frac{|C_{\alpha}|^2}{C_{\phi,\psi,\alpha}}\int_{\mathbb{R}_{0}^n}\int_{\mathbb{R}^n}\phi_{\textbf{\em a},\boldsymbol b,\alpha}(\boldsymbol t)&(W_{\psi,\alpha}f)(\textbf{\em b},\boldsymbol a)\frac{d\boldsymbol b d\textbf{\em a}}{|\textbf{\em a}|_{p}^2}\notag\\
&=\frac{|C_{\alpha}|^2}{C_{\phi,\psi,\alpha}}\int_{\mathbb{R}_{0}^n}\left(\int_{\mathbb{R}^n}\overline{(\mathfrak{F}_{\alpha ,\boldsymbol b}\{\overline{\phi_{\textbf{\em a},\boldsymbol b,\alpha}(\boldsymbol t)}\})(\boldsymbol\xi)}(\mathfrak{F}_{\alpha,\boldsymbol b}\{(W_{\psi,\alpha}f)(\textbf{\em b},\boldsymbol a)\})(\boldsymbol\xi)d\boldsymbol\xi\right)\frac{d\textbf{\em a}}{|\textbf{\em a}|_{p}^2}.
\end{align}
Using lemma \ref{P1lemma2.2} and lemma \ref{P1lemma2.3}, we obtain
\begin{equation*}
\begin{array}{l}
\frac{|C_{\alpha}|^2}{C_{\phi,\psi,\alpha}}\displaystyle\int_{\mathbb{R}_{0}^n}\int_{\mathbb{R}^n}\phi_{\textbf{\em a},\boldsymbol b,\alpha}(\boldsymbol t)(W_{\psi,\alpha}f)(\textbf{\em b},\boldsymbol a)\frac{d\boldsymbol b d\textbf{\em a}}{|\textbf{\em a}|_{p}^2}\\ [8pt]
=\frac{|C_{\alpha}|^2}{C_{\phi,\psi,\alpha}}\displaystyle\int_{\mathbb{R}_{0}^n}\bigg(\int_{\mathbb{R}^n}\frac{|\textbf{\em a}|_{p}}{C_{\alpha}}e^{-\frac{i}{2}(\|\boldsymbol t\|^2+\|\boldsymbol\xi\|^2)\cot\alpha+i\langle\boldsymbol t,\boldsymbol\xi\rangle\csc\alpha}\left(\mathfrak{F}_\alpha\left\{e^{-\frac{1}{2}\|\cdot\|^2\cot\alpha}\phi(\cdot)\right\}\right)(\textbf{\em a}\boldsymbol\xi) \\[8pt] \hspace{3 cm}\times\overline{\left(\mathfrak{F}_\alpha\left\{e^{-\frac{1}{2}\|\cdot\|^2\cot\alpha}\psi(\cdot)\right\}\right)(\textbf{\em a}\boldsymbol\xi)}(\mathfrak{F}_\alpha f)(\boldsymbol\xi)d\boldsymbol\xi\bigg)\frac{d\textbf{\em a}}{|\textbf{\em a}|_{p}^2}\\[8pt]
=\frac{|C_{\alpha}|^2}{C_{\phi,\psi,\alpha}}\displaystyle\int_{\mathbb{R}^n}\frac{1}{C_{\alpha}}e^{-\frac{i}{2}(\|\boldsymbol t\|^2+\|\boldsymbol\xi\|^2)\cot\alpha+i\langle\boldsymbol t,\boldsymbol\xi\rangle\csc\alpha}(\mathfrak{F}_\alpha f)(\boldsymbol\xi)\\[8pt] \hspace{3 cm}\times\left(\displaystyle\int_{\mathbb{R}_{0}^n}\left(\mathfrak{F}_\alpha\left\{e^{-\frac{1}{2}\|\cdot\|^2\cot\alpha}\phi(\cdot)\right\}\right)(\textbf{\em a}\boldsymbol\xi)\overline{\left(\mathfrak{F}_\alpha\left\{e^{-\frac{1}{2}\|\cdot\|^2\cot\alpha}\psi(\cdot)\right\}\right)(\textbf{\em a}\boldsymbol\xi)}\frac{1}{|\textbf{\em a}|_{p}}d\textbf{\em a}\right)d\boldsymbol\xi.
\end{array}
\end{equation*}
Using equation (\ref{P1eqn14}) and equation (\ref{P1eqn16}), yields
\begin{eqnarray}\label{P1eqn18}
\frac{|C_{\alpha}|^2}{C_{\phi,\psi,\alpha}}\int_{\mathbb{R}_{0}^n}\int_{\mathbb{R}^n}\phi_{\textbf{\em a},\boldsymbol b,\alpha}(\boldsymbol t)(W_{\psi,\alpha}f)(\textbf{\em b},\boldsymbol a)\frac{d\boldsymbol b d\textbf{\em a}}{|\textbf{\em a}|_{p}^2}&=&|C_{\alpha}|^2\int_{\mathbb{R}^n}\frac{1}{C_{\alpha}}e^{-\frac{i}{2}(\|\boldsymbol t\|^2+\|\boldsymbol\xi\|^2)\cot\alpha+i\langle\boldsymbol t,\boldsymbol\xi\rangle\csc\alpha}(\mathfrak{F}_\alpha f)(\boldsymbol\xi)d\boldsymbol\xi\notag\\
&=&\int_{\mathbb{R}^n}\overline{C_{\alpha}}e^{-\frac{i}{2}(\|\boldsymbol t\|^2+\|\boldsymbol\xi\|^2)\cot\alpha+i\langle\boldsymbol t,\boldsymbol\xi\rangle\csc\alpha}(\mathfrak{F}_\alpha f)(\boldsymbol\xi)d\boldsymbol\xi\notag\\
&=&\int_{\mathbb{R}^n}K_{-\alpha}(\boldsymbol t,\boldsymbol\xi)(\mathfrak{F}_\alpha f)(\boldsymbol\xi)d\boldsymbol\xi\notag\\
&=&f(\boldsymbol t).
\end{eqnarray}
The theorem follows from equation (\ref{P1eqn17}) and equation (\ref{P1eqn18}).
\end{proof}
\begin{corollary}\label{P1corollary2.3}
For a given $f\in L^2(\mathbb{R}^n)$ and a wavelet $\psi$ we also have the following reconstruction formula:
$$f(\boldsymbol t)=\frac{|C_{\alpha}|^2}{C_{\psi,\alpha}}\int_{\mathbb{R}_{0}^n}\int_{\mathbb{R}^n}\psi_{\textbf{\em a},\boldsymbol b,\alpha}(\boldsymbol t)(W_{\psi,\alpha}f)(\textbf{\em b},\boldsymbol a)\frac{d\boldsymbol b d\textbf{\em a}}{|\textbf{\em a}|_{p}^2}.$$
\end{corollary}
\subsection{Reproducing Kernel function involving two wavelet for the range of CFrWT} 
\begin{theorem}\label{P1theo2.4}
Let $\phi$ and $\psi$ be two wavelets satisfying (\ref{P1eqn13}). Also let $C_{\phi,\psi,\alpha},$ as defined in (\ref{P1eqn14}), is non-zero and $(\boldsymbol b_{0},\boldsymbol a_{0})$ be a point in $\mathbb{R}^n\times\mathbb{R}_{0}^n.$ Then, $F\in L^2\left(\mathbb{R}^n\times\mathbb{R}_{0}^n,\frac{d\boldsymbol b d\boldsymbol a}{|\boldsymbol a|^2_{p}}\right)$ is the CFrWT of some $f\in L^2(\mathbb{R}^n)$ if and only if
$$F(\boldsymbol b_{0},\boldsymbol a_{0})=\int_{\mathbb{R}_{0}^n}\int_{\mathbb{R}^n}F(\boldsymbol b,\boldsymbol a)K_{\phi,\psi,\alpha}(\boldsymbol b_{0},\boldsymbol a_{0};\boldsymbol b,\boldsymbol a)\frac{d\boldsymbol b d\boldsymbol a}{|\boldsymbol a|^2_{p}},$$
where $K_{\phi,\psi,\alpha}(\boldsymbol b_{0},\boldsymbol a_{0};\boldsymbol b,\boldsymbol a)$ is the reproducing kernel, satisfying
$$K_{\phi,\psi,\alpha}(\boldsymbol b_{0},\boldsymbol a_{0};\boldsymbol b,\boldsymbol a)=\frac{|C_{\alpha}|^2}{C_{\phi,\psi,\alpha}}\int_{\mathbb{R}^n}\phi_{\boldsymbol a,\boldsymbol b,\alpha}(\boldsymbol t)\overline{{\psi_{\boldsymbol a_{0},\boldsymbol b_{0},\alpha}}(\boldsymbol t)}d\boldsymbol t.$$
\end{theorem}
\begin{proof}
Let $f\in L^2(\mathbb{R}^n)$ such that $W_{\psi,\alpha}f=F,$ then
\begin{eqnarray*}
F(\boldsymbol b_{0},\boldsymbol a_{0})&=&(W_{\psi,\alpha}f)(\boldsymbol b_{0},\boldsymbol a_{0})\\
&=&\int_{\mathbb{R}^n}f(\boldsymbol t)\overline{\psi_{\boldsymbol a_{0},\boldsymbol b_{0},\alpha}(\boldsymbol t)}d\boldsymbol t\\
&=&\frac{|C_{\alpha}|^2}{C_{\phi,\psi,\alpha}}\int_{\mathbb{R}^n}\left(\int_{\mathbb{R}_{0}^n}\int_{\mathbb{R}^n}\phi_{\textbf{\em a},\boldsymbol b,\alpha}(\boldsymbol t)(W_{\psi,\alpha}f)(\textbf{\em b},\boldsymbol a)\frac{d\boldsymbol b d\textbf{\em a}}{|\textbf{\em a}|_{p}^2} \right)\overline{\psi_{\boldsymbol a_{0},\boldsymbol b_{0},\alpha}(\boldsymbol t)}d\boldsymbol t\\
&=&\int_{\mathbb{R}_{0}^n}\int_{\mathbb{R}^n}(W_{\psi,\alpha}f)(\textbf{\em b},\boldsymbol a)\left(\frac{|C_{\alpha}|^2}{C_{\phi,\psi,\alpha}}\int_{\mathbb{R}^n}\phi_{\boldsymbol a,\boldsymbol b,\alpha}(\boldsymbol t)\overline{{\psi_{\boldsymbol a_{0},\boldsymbol b_{0},\alpha}}(\boldsymbol t)}d\boldsymbol t\right)\frac{d\boldsymbol b d\boldsymbol a}{|\boldsymbol a|^2_{p}}\\
&=&\int_{\mathbb{R}_{0}^n}\int_{\mathbb{R}^n}F(\boldsymbol b,\boldsymbol a)K_{\phi,\psi,\alpha}(\boldsymbol b_{0},\boldsymbol a_{0};\boldsymbol b,\boldsymbol a)\frac{d\boldsymbol b d\boldsymbol a}{|\boldsymbol a|^2_{p}},
\end{eqnarray*}
where $K_{\phi,\psi,\alpha}(\boldsymbol b_{0},\boldsymbol a_{0};\boldsymbol b,\boldsymbol a)=\frac{|C_{\alpha}|^2}{C_{\phi,\psi,\alpha}}\ \displaystyle\int_{\mathbb{R}^n}\phi_{\boldsymbol a,\boldsymbol b,\alpha}(\boldsymbol t)\overline{{\psi_{\boldsymbol a_{0},\boldsymbol b_{0},\alpha}}(\boldsymbol t)}d\boldsymbol t.$\\[8pt]
Conversely, let for the given $F\in L^2\left(\mathbb{R}^n\times\mathbb{R}_{0}^n,\frac{d\boldsymbol b d\boldsymbol a}{|\boldsymbol a|^2_{p}}\right)$ the condition holds.\\[8pt]
 Then the required $f$ is given by $\displaystyle\int_{{R}_{0}^n}\int_{{R}^n} F(\boldsymbol b,\textbf{\em a})\phi_{\textbf{\em a},\boldsymbol b,\alpha}\frac{d\boldsymbol b d\boldsymbol a}{|\boldsymbol a|^2_{p}}.$
\end{proof}
\begin{corollary}\label{P1corollary2.4}
Let $\psi$ be a wavelet and $(\boldsymbol b_{0},\boldsymbol a_{0})$ be a point in $\mathbb{R}^n\times\mathbb{R}_{0}^n.$ Then, $F\in L^2\left(\mathbb{R}^n\times\mathbb{R}_{0}^n,\frac{d\boldsymbol b d\boldsymbol a}{|\boldsymbol a|^2_{p}}\right)$ is the CFrWT of some $f\in L^2(\mathbb{R}^n)$ if and only if
$$F(\boldsymbol b_{0},\boldsymbol a_{0})=\int_{\mathbb{R}_{0}^n}\int_{\mathbb{R}^n}F(\boldsymbol b,\boldsymbol a)K_{\psi,\alpha}(\boldsymbol b_{0},\boldsymbol a_{0};\boldsymbol b,\boldsymbol a)\frac{d\boldsymbol b d\boldsymbol a}{|\boldsymbol a|^2_{p}},$$
where $K_{\psi,\alpha}(\boldsymbol b_{0},\boldsymbol a_{0};\boldsymbol b,\boldsymbol a)$ is the reproducing kernel, satisfying
$$K_{\psi,\alpha}(\boldsymbol b_{0},\boldsymbol a_{0};\boldsymbol b,\boldsymbol a)=\frac{|C_{\alpha}|^2}{C_{\psi,\alpha}}\int_{\mathbb{R}^n}\psi_{\boldsymbol a,\boldsymbol b,\alpha}(\boldsymbol t)\overline{{\psi_{\boldsymbol a_{0},\boldsymbol b_{0},\alpha}}(\boldsymbol t)}d\boldsymbol t.$$
\end{corollary}
\begin{proof}
Taking $\phi=\psi$ in the above theorem, the corollary follows.
\end{proof}
\section{Important Inequalities for CFrWT}
\subsection{Heisenberg's uncertainty inequality for CFrWT}
 If $f\in L^2(\mathbb{R}^n),$ then we have the Heisenberg's uncertainty inequality in $n$-dimensions as (\cite{folland1997uncertainty})
 $$\left(\int_{\mathbb{R}^n}\|\boldsymbol t\|^2|f(\boldsymbol t)|^2d\boldsymbol t\right)\left(\int_{\mathbb{R}^n}\|\boldsymbol \xi\|^2|(\mathfrak{F}f)(\boldsymbol \xi)|^2d\boldsymbol \xi\right)\geq\frac{n^2}{4}\|f\|_{L^2(\mathbb{R}^n)}^4,$$
 where $\mathfrak{F}f$ denotes the classical Fourier transform of $f$.\\
 
 Let $f\in L^2(\mathbb{R}^n)$ and none of $\alpha,\gamma$ and $\alpha-\gamma$ are integral multiple of $\pi.$
 Let $H(\boldsymbol\xi)=(\mathfrak{F}_\alpha f)(\boldsymbol\xi)e^{-\frac{i}{2}\|\boldsymbol\xi\|^2\cot\gamma}$ and $h(\boldsymbol t)=\frac{1}{(2\pi)^{\frac{n}{2}}}\displaystyle\int_{\mathbb{R}^n}H(\boldsymbol\xi)e^{i\langle\boldsymbol\xi,\boldsymbol t\rangle}d\boldsymbol\xi.$ Then the Heisenberg's uncertainty inequality for the function $h$ is
 $$\left(\int_{\mathbb{R}^n}\|\boldsymbol t\|^2|h(\boldsymbol t)|^2d\boldsymbol t\right)\left(\int_{\mathbb{R}^n}\|\boldsymbol \xi\|^2|(\mathfrak{F}h)(\boldsymbol \xi)|^2d\boldsymbol \xi\right)\geq\frac{n^2}{4}\|h\|_{L^2(\mathbb{R}^n)}^4.$$
 This implies,
 \begin{equation}\label{P1eqn19}
\left(\int_{\mathbb{R}^n}\|\boldsymbol t\|^2|h(\boldsymbol t)|^2d\boldsymbol t\right)\left(\int_{\mathbb{R}^n}\|\boldsymbol \xi\|^2|(\mathfrak{F_{\alpha}}f)(\boldsymbol \xi)|^2d\boldsymbol \xi\right)\geq\frac{n^2}{4}\|f\|_{L^2(\mathbb{R}^n)}^4.
\end{equation}
Now,
\begin{eqnarray}\label{P1eqn20}
\int_{\mathbb{R}^n}\|\boldsymbol t\|^2|h(\boldsymbol t)|^2d\boldsymbol t &=&|\csc\gamma|^n\int_{\mathbb{R}^n}\|(\csc\gamma)\boldsymbol t\|^2|h((\csc\gamma)\boldsymbol t)|^2d\boldsymbol t.
\end{eqnarray}
 We have,
 \begin{eqnarray}\label{P1eqn21}
 |h((\csc\gamma)\boldsymbol t)|^2&=&\left|\frac{1}{(2\pi)^{\frac{n}{2}}}\int_{\mathbb{R}^n}H(\boldsymbol\xi)e^{i\langle\boldsymbol\xi,(\csc\gamma)\boldsymbol t\rangle}d\boldsymbol\xi\right|^2\notag\\
 &=&\left|\frac{1}{(2\pi)^{\frac{n}{2}}}\int_{\mathbb{R}^n}H(\boldsymbol\xi)e^{i\langle\boldsymbol\xi,\boldsymbol t\rangle\csc\gamma}d\boldsymbol\xi\right|^2\notag\\
 &=&\left|\frac{1}{(2\pi)^{\frac{n}{2}}}\int_{\mathbb{R}^n}(\mathfrak{F}_\alpha f)(\boldsymbol\xi)e^{-\frac{i}{2}\|\boldsymbol\xi\|^2\cot\gamma}e^{i\langle\boldsymbol\xi,\boldsymbol t\rangle\csc\gamma}d\boldsymbol\xi\right|^2\notag\\
 &=&\left|\frac{1}{(2\pi)^{\frac{n}{2}}\overline{C_{\gamma}}}\int_{\mathbb{R}^n}(\mathfrak{F}_\alpha f)(\boldsymbol\xi)K_{-\gamma}(\boldsymbol t,\boldsymbol\xi)d\boldsymbol\xi\right|^2\notag\\
 &=&\frac{1}{(2\pi)^{n}|C_{\gamma}|^2}|(\mathfrak{F}_{\alpha-\gamma} f)(\boldsymbol t)|^2.
\end{eqnarray} 
 Using equation (\ref{P1eqn21}) in equation (\ref{P1eqn20}), we obtain
$$\int_{\mathbb{R}^n}\|\boldsymbol t\|^2|h(\boldsymbol t)|^2d\boldsymbol t=\int_{\mathbb{R}^n}\|(\csc\gamma)\boldsymbol t\|^2|(\mathfrak{F}_{\alpha-\gamma} f)(\boldsymbol t)|^2d\boldsymbol t.$$
Let $\beta=\alpha-\gamma,$ then 
\begin{equation}\label{P1eqn22}
 \int_{\mathbb{R}^n}\|\boldsymbol t\|^2|h(\boldsymbol t)|^2d\boldsymbol t=\int_{\mathbb{R}^n}\|\csc{(\alpha-\beta})\boldsymbol t\|^2|(\mathfrak{F}_\beta f)(\boldsymbol t)|^2d\boldsymbol t.
 \end{equation}
 From equation (\ref{P1eqn19}) and equation (\ref{P1eqn22}), we have
 \begin{equation}\label{P1eqn23}
\left(\int_{\mathbb{R}^n}\|\boldsymbol t\|^2|(\mathfrak{F}_\beta f)(\boldsymbol t)|^2d\boldsymbol t\right)\left(\int_{\mathbb{R}^n}\|\boldsymbol \xi\|^2|(\mathfrak{F_{\alpha}}f)(\boldsymbol \xi)|^2d\boldsymbol \xi\right)\geq\frac{n^2}{4}|\sin{(\alpha-\beta})|^2\|f\|_{L^2(\mathbb{R}^n)}^4.
\end{equation}
Equation (\ref{P1eqn23}) is known as the $n$ dimensional  Heisenberg's uncertainty inequality in two fractional Fourier domain. For the case in one dimension refer \cite{guanlei2009logarithmic}.\\
To prove the main result of this subsection we need the following lemma.
 \begin{lemma}\label{P1lemma3.1}
 Let $\psi$ be a wavelet and $f\in L^2(\mathbb{R}^n)$. Then
 $$\int_{\mathbb{R}_{0}^n}\int_{\mathbb{R}^n}\|\boldsymbol\xi\|^2|(\mathfrak{F}_{\alpha,\boldsymbol b}\{(W_{\psi,\alpha}f)(\boldsymbol b,\boldsymbol a)\})(\boldsymbol\xi)|^2d\boldsymbol\xi\frac{d\boldsymbol a}{|\boldsymbol a|_{p}^2}=\frac{C_{\psi,\alpha}}{|C_{\alpha}|^2}\int_{\mathbb{R}^n}\|\boldsymbol\xi\|^2|(\mathfrak{F}_{\alpha}f)(\boldsymbol\xi)|^2d\boldsymbol\xi,$$
 where $C_{\psi,\alpha}$ is as defined in (\ref{P1eqn5}) and $\mathfrak{F}_{\alpha,\boldsymbol b}$ denotes the FrFT to be taken with  respect to the argument $\boldsymbol b.$
 \end{lemma}
 \begin{proof}
 Using lemma \ref{P1lemma2.3}, we have
 \begin{equation*}
 \begin{array}{l}
 \hspace{-3cm}\displaystyle\int_{\mathbb{R}_{0}^n}\int_{\mathbb{R}^n}\|\boldsymbol\xi\|^2|(\mathfrak{F}_{\alpha,\boldsymbol b}\{(W_{\psi,\alpha}f)(\boldsymbol b,\boldsymbol a)\})(\boldsymbol\xi)|^2d\boldsymbol\xi\frac{d\boldsymbol a}{|\boldsymbol a|_{p}^2}\\
 =\frac{1}{|C_{\alpha}|^2}\displaystyle\int_{\mathbb{R}^n}\|\boldsymbol\xi\|^2|(\mathfrak{F}_{\alpha}f)(\boldsymbol\xi)|^2\left(\int_{\mathbb{R}_{0}^n}\left|\left(\mathfrak{F}_{\alpha}\left\{e^{-\frac{i}{2}\|\cdot\|^2\cot\alpha}\psi(\cdot)\right\}\right)(\boldsymbol a\boldsymbol\xi)\right|^2\frac{1}{|\boldsymbol a|_{p}}d\boldsymbol a\right)d\boldsymbol\xi\\
 =\frac{C_{\psi,\alpha}}{|C_{\alpha}|^2}\displaystyle\int_{\mathbb{R}^n}\|\boldsymbol\xi\|^2|(\mathfrak{F}_{\alpha}f)(\boldsymbol\xi)|^2d\boldsymbol\xi.
  \end{array}
 \end{equation*}
This completes the proof.
 \end{proof}
 Now we give the main result of this subsection.
 \begin{theorem}\label{P1theo3.1}
(Heisenberg's Uncertainty Inequality for CFrWT) Let $\psi$ be a wavelet and $f\in L^2(\mathbb{R}^n).$ Let $(W_{\psi,\alpha}f)(\boldsymbol b,\boldsymbol a)$ be the CFrWT of $f$ with respect  to the parameter $\alpha.$ Then
 $$\left(\int_{\mathbb{R}^n}\|\boldsymbol t\|^2|(\mathfrak{F}_{\beta,\boldsymbol b}\{(W_{\psi,\alpha}f)(\boldsymbol b,\boldsymbol a)\})(\boldsymbol t)|^2d\boldsymbol t\frac{d\boldsymbol a}{|\boldsymbol a|^2_{p}}\right)\left(\int_{\mathbb{R}^n}\|\boldsymbol\xi\|^2|(\mathfrak{F}_{\alpha}f)(\boldsymbol\xi)|^2d\boldsymbol\xi\right)\geq\frac{n^2C_{\psi,\alpha}}{4|C_{\alpha}|^2}|\sin(\alpha-\beta)|^2\|f\|^4_{L^2(\mathbb{R}^n)}. $$
 \end{theorem}
 \begin{proof}
 Replacing $f$ by $(W_{\psi,\alpha}f)(\boldsymbol b,\boldsymbol a)$ in equation (\ref{P1eqn23}), we obtain
 \begin{align*}
\left(\int_{\mathbb{R}^n_{0}}\int_{\mathbb{R}^n}\|\boldsymbol t\|^2|(\mathfrak{F}_{\beta,\boldsymbol b}\{(W_{\psi,\alpha}f)(\boldsymbol b,\boldsymbol a)\})(\boldsymbol t)|^2d\boldsymbol t\right)&\left(\int_{\mathbb{R}^n}\|\boldsymbol\xi\|^2|(\mathfrak{F}_{\alpha,\boldsymbol b}\{(W_{\psi,\alpha}f)(\boldsymbol b,\boldsymbol a)\})(\boldsymbol\xi)|^2d\boldsymbol\xi\right)\\
&\geq \frac{n^2}{4}|\sin(\alpha-\beta)|^2\|(W_{\psi,\alpha}f)(\boldsymbol b,\boldsymbol a)\|^4_{L^2(\mathbb{R}^n)}.
   \end{align*}
  Taking square root on both sides, we get
\begin{align*} 
  \left(\int_{\mathbb{R}^n}\|\boldsymbol t\|^2|(\mathfrak{F}_{\beta,\boldsymbol b}\{(W_{\psi,\alpha}f)(\boldsymbol b,\boldsymbol a)\})(\boldsymbol t)|^2d\boldsymbol t\right)^{\frac{1}{2}}&\left(\int_{\mathbb{R}^n}\|\boldsymbol\xi\|^2|(\mathfrak{F}_{\alpha,\boldsymbol b}\{(W_{\psi,\alpha}f)(\boldsymbol b,\boldsymbol a)\})(\boldsymbol\xi)|^2d\boldsymbol\xi\right)^{\frac{1}{2}}\\ 
 &\geq \frac{n}{2}|\sin(\alpha-\beta)|\|(W_{\psi,\alpha}f)(\cdot,\boldsymbol a)\|^2_{L^2(\mathbb{R}^n)}\\
 &= \frac{n}{2}|\sin(\alpha-\beta)|\int_{\mathbb{R}^n}|(W_{\psi,\alpha}f)(\boldsymbol b,\boldsymbol a)|^2d\boldsymbol b.
 \end{align*} 
 Integrating both sides with respect to the measure $\frac{d\boldsymbol a}{|\boldsymbol a|_{p}^2},$ we have\\
  \begin{align*}
 \int_{\mathbb{R}_{0}^n}\left(\int_{\mathbb{R}^n}\|\boldsymbol t\|^2|(\mathfrak{F}_{\beta,\boldsymbol b}\{(W_{\psi,\alpha}f)(\boldsymbol b,\boldsymbol a)\})(\boldsymbol t)|^2d\boldsymbol t\right)^{\frac{1}{2}}&\left(\int_{\mathbb{R}^n}\|\boldsymbol\xi\|^2|(\mathfrak{F}_{\alpha,\boldsymbol b}\{(W_{\psi,\alpha}f)(\boldsymbol b,\boldsymbol a)\})(\boldsymbol\xi)|^2d\boldsymbol\xi\right)^{\frac{1}{2}}\frac{d\boldsymbol a}{|\boldsymbol a|^2_{p}}\\ 
 &\geq \frac{n}{2}|\sin(\alpha-\beta)|\|W_{\psi,\alpha}f\|^2_{L^2(\mathbb{R}^n\times\mathbb{R}_{0}^n,\frac{d\boldsymbol b d\boldsymbol a}{|\boldsymbol a|^2_{p}})}\\
 &= \frac{nC_{\psi,\alpha}}{2|C_{\alpha}|^2}|\sin(\alpha-\beta)|\|f\|^2_{L^2(\mathbb{R}^n)}.
 \end{align*}
 Using Holder's inequality, we have
\begin{eqnarray*}
 \left(\int_{\mathbb{R}_{0}^n}\int_{\mathbb{R}^n}\|\boldsymbol t\|^2|(\mathfrak{F}_{\beta,\boldsymbol b}\{(W_{\psi,\alpha}f)(\boldsymbol b,\boldsymbol a)\})(\boldsymbol t)|^2d\boldsymbol t\frac{d\boldsymbol a}{|\boldsymbol a|^2_{p}}\right)^{\frac{1}{2}}&\left(\displaystyle\int_{\mathbb{R}_{0}^n}\int_{\mathbb{R}^n}\|\boldsymbol\xi\|^2|(\mathfrak{F}_{\alpha,\boldsymbol b}\{(W_{\psi,\alpha}f)(\boldsymbol b,\boldsymbol a)\})(\boldsymbol\xi)|^2d\boldsymbol\xi\frac{d\boldsymbol a}{|\boldsymbol a|^2_{p}}\right)^{\frac{1}{2}}\\ 
 &\geq \frac{nC_{\psi,\alpha}}{2|C_{\alpha}|^2}|\sin(\alpha-\beta)|\|f\|^2_{L^2(\mathbb{R}^n)}.
 \end{eqnarray*} 
 Squaring both sides, we obtain
 \begin{align*}
 \left(\int_{\mathbb{R}_{0}^n}\int_{\mathbb{R}^n}\|\boldsymbol t\|^2|(\mathfrak{F}_{\beta,\boldsymbol b}\{(W_{\psi,\alpha}f)(\boldsymbol b,\boldsymbol a)\})(\boldsymbol t)|^2d\boldsymbol t\frac{d\boldsymbol a}{|\boldsymbol a|^2_{p}}\right)&\left(\int_{\mathbb{R}_{0}^n}\int_{\mathbb{R}^n}\|\boldsymbol\xi\|^2|(\mathfrak{F}_{\alpha,\boldsymbol b}\{(W_{\psi,\alpha}f)(\boldsymbol b,\boldsymbol a)\})(\boldsymbol\xi)|^2d\boldsymbol\xi\frac{d\boldsymbol a}{|\boldsymbol a|^2_{p}}\right)\\
&\geq \frac{n^2C_{\psi,\alpha}^2}{4|C_{\alpha}|^4}|\sin(\alpha-\beta)|^2\|f\|^4_{L^2(\mathbb{R}^n)}.
\end{align*}
 Using lemma \ref{P1lemma3.1}, we have
 $$ \left(\int_{\mathbb{R}_{0}^n}\int_{\mathbb{R}^n}\|\boldsymbol t\|^2|(\mathfrak{F}_{\beta,\boldsymbol b}\{(W_{\psi,\alpha}f)(\boldsymbol b,\boldsymbol a)\})(\boldsymbol t)|^2d\boldsymbol t\frac{d\boldsymbol a}{|\boldsymbol a|^2_{p}}\right)\left(\frac{C_{\psi,\alpha}}{|C_{\alpha}|^2}\int_{\mathbb{R}^n}\|\boldsymbol\xi\|^2|(\mathfrak{F}_{\alpha}f)(\boldsymbol\xi)|^2d\boldsymbol\xi\right)\geq\frac{n^2C_{\psi,\alpha}^2}{4|C_{\alpha}|^4}|\sin(\alpha-\beta)|^2\|f\|^4_{L^2(\mathbb{R}^n)}.$$
 Therefore,
$$\left(\int_{\mathbb{R}_{0}^n}\int_{\mathbb{R}^n}\|\boldsymbol t\|^2|(\mathfrak{F}_{\beta,\boldsymbol b}\{(W_{\psi,\alpha}f)(\boldsymbol b,\boldsymbol a)\})(\boldsymbol t)|^2d\boldsymbol t\frac{d\boldsymbol a}{|\boldsymbol a|^2_{p}}\right)\left(\int_{\mathbb{R}^n}\|\boldsymbol\xi\|^2|(\mathfrak{F}_{\alpha}f)(\boldsymbol\xi)|^2d\boldsymbol\xi\right)\geq|\sin(\alpha-\beta)|^2\frac{n^2C_{\psi,\alpha}}{4|C_{\alpha}|^2}\|f\|^4_{L^2(\mathbb{R}^n)}.$$
This completes the proof.
 \end{proof}
 The corollary below is a Heisenberg's uncertainty inequality for classical wavelet transform in $n$-dimensions.
 \begin{corollary}
 Let $(W_{\psi}f)(\boldsymbol b,\boldsymbol a)$ be the classical wavelet transform of $f\in L^2(\mathbb{R}^n)$ with respect  to the wavelet $\psi.$ Then
 $$\left(\int_{\mathbb{R}^n}\|\boldsymbol t\|^2|(W_{\psi}f)(\boldsymbol t,\boldsymbol a)|^2d\boldsymbol t\frac{d\boldsymbol a}{|\boldsymbol a|^2_{p}}\right)\left(\int_{\mathbb{R}^n}\|\boldsymbol\xi\|^2|(\mathfrak{F}f)(\boldsymbol\xi)|^2d\boldsymbol\xi\right)\geq\frac{(2\pi)^n n^2C_{\psi}}{4}\|f\|^4_{L^2(\mathbb{R}^n)}, $$
 where $C_\psi=\displaystyle\int_{\mathbb{R}_0^n}\frac{|(\mathfrak{F}\psi)(\boldsymbol u)|}{|\boldsymbol u|_{p}}d\boldsymbol u.$
 
 \end{corollary}
 \subsection{Local Uncertainty Inequality for CFrWT}
 With the definition of classical Fourier transform for $\alpha=\frac{\pi}{2}$, the local uncertainty principle can be stated as follows (\cite{folland1997uncertainty}):
 \begin{itemize}
 \item[(i)] For $0<\theta<\frac{n}{2},$ there exists a constant $A_{\theta}$ such that for all measurable subset $E$ of $\mathbb{R}^n$ and all $f\in L^2(\mathbb{R}^n),$
 $$\int_E|(\mathfrak{F}f)(\boldsymbol\xi)|^2d\boldsymbol\xi\leq A_{\theta}(\lambda(E))^\frac{2\theta}{n}\|\|\boldsymbol x\|^{\theta}f\|^2_{L^2(\mathbb{R}^n)}.$$
 \item[(ii)] For $\theta>\frac{n}{2},$ there exists a constant $A_{\theta}$ such that for all measurable subset $E$ of $\mathbb{R}^n$ and all $f\in L^2(\mathbb{R}^n),$
 $$\int_E|(\mathfrak{F}f)(\boldsymbol\xi)|^2d\boldsymbol\xi\leq A_{\theta}\lambda(E)\|f\|^{2-\frac{n}{\theta}}_{L^2(\mathbb{R}^n)}\|\|\boldsymbol x\|^{\theta}f\|^{\frac{n}{\theta}}_{L^2(\mathbb{R}^n)},$$
 where $\lambda(E)$ denotes the Lebesgue measure of $E.$
\end{itemize}  
 Let $f\in L^2(\mathbb{R}^n)$ and none of $\alpha,\gamma$ and $\alpha-\gamma$ are integral multiple of $\pi.$\\
Let $H(\boldsymbol\xi)=(\mathfrak{F}_\alpha f)(\boldsymbol\xi)e^{-\frac{i}{2}\|\boldsymbol\xi\|^2\cot\gamma}$ and $h(\boldsymbol t)=\frac{1}{(2\pi)^{\frac{n}{2}}}\displaystyle\int_{\mathbb{R}^n}H(\boldsymbol\xi)e^{i\langle\boldsymbol\xi,\boldsymbol t\rangle}d\boldsymbol\xi.$ Then 
\begin{equation}\label{P1eqn24}
 \int_E|(\mathfrak{F}h)(\boldsymbol \xi)|^2d\boldsymbol \xi=\int_E|(\mathfrak{F_{\alpha}}f)(\boldsymbol \xi)|^2d\boldsymbol \xi.
\end{equation}
Now,
\begin{eqnarray}\label{P1eqn25}
\int_{\mathbb{R}^n}\|\boldsymbol t\|^{2\theta}|h(\boldsymbol t)|^2d\boldsymbol t &=&|\csc\gamma|^n\int_{\mathbb{R}^n}\|(\csc\gamma)\boldsymbol t\|^{2\theta}|h((\csc\gamma)\boldsymbol t)|^2d\boldsymbol t.
\end{eqnarray}
 Using equation (\ref{P1eqn21}) in equation (\ref{P1eqn25}), yields
 $$\int_{\mathbb{R}^n}\|\boldsymbol t\|^{2\theta}|h(\boldsymbol t)|^2d\boldsymbol t=\int_{\mathbb{R}^n}\|(\csc\gamma)\boldsymbol t\|^{2\theta}\big|(\mathfrak{F}_{\alpha-\gamma} f)(\boldsymbol t)\big|^2d\boldsymbol t.$$
Putting $\beta=\alpha-\gamma$, we obtain
\begin{equation}\label{P1eqn26}
\int_{\mathbb{R}^n}\|\boldsymbol t\|^{2\theta}|h(\boldsymbol t)|^2d\boldsymbol t=\int_{\mathbb{R}^n}\|\csc{(\alpha-\beta})\boldsymbol t\|^{2\theta}\big|(\mathfrak{F}_\beta f)(\boldsymbol t)\big|^2d\boldsymbol t.
\end{equation} 
 \begin{lemma}\label{P1lemma4.1}
(Local uncertainty inequality in two fractional Fourier transform domain) For a given positive $\theta\neq\frac{n}{2}$ there exists a constant $A_{\theta}$ such that for all measurable subset $E$ of $\mathbb{R}^n$ and all $f\in L^2(\mathbb{R}^n),$ 
 $$ \int_E|(\mathfrak{F}_{\alpha}f)(\boldsymbol\xi)|^2d\boldsymbol\xi\leq \frac{A_{\theta}}{|\sin(\alpha-\beta)|^{2\theta}}(\lambda(E))^\frac{2\theta}{n}\|\|\boldsymbol x\|^{\theta}(\mathfrak{F}_{\beta}f)\|^2_{L^2(\mathbb{R}^n)},\ \mbox{if}\ 0<\theta<\frac{n}{2}$$ and
 $$\int_E|(\mathfrak{F}_{\alpha}f)(\boldsymbol\xi)|^2d\boldsymbol\xi\leq \frac{A_{\theta}}{|\sin(\alpha-\beta)|^{n}}\lambda(E)\|f\|^{2-\frac{n}{\theta}}_{L^2(\mathbb{R}^n)}\|\|\boldsymbol x\|^{\theta}(\mathfrak{F}_{\beta}f)\|^{\frac{n}{\theta}}_{L^2(\mathbb{R}^n)},\ \mbox{if}\ \theta>\frac{n}{2}.$$
 \end{lemma}
 \begin{proof}
Consider the function $h$ defined by
$$h(\boldsymbol t)=\frac{1}{(2\pi)^{\frac{n}{2}}}\int_{\mathbb{R}^n}H(\boldsymbol\xi)e^{i\langle\boldsymbol\xi,\boldsymbol t\rangle}d\boldsymbol\xi,\ \mbox{where}\  H(\boldsymbol\xi)=(\mathfrak{F}_\alpha f)(\boldsymbol\xi)e^{-\frac{i}{2}\|\boldsymbol\xi\|^2\cot\gamma}.$$ 
We have, for $0<\theta<\frac{n}{2}$
\begin{equation}\label{P1eqn27}
\int_E|(\mathfrak{F}h)(\boldsymbol\xi)|^2d\boldsymbol\xi\leq A_{\theta}(\lambda(E))^\frac{2\theta}{n}\|\|\boldsymbol x\|^{\theta}h\|^2_{L^2(\mathbb{R}^n)}.
\end{equation}
From (\ref{P1eqn24}), (\ref{P1eqn26}) and (\ref{P1eqn27}) , we get
$$\int_E|(\mathfrak{F_{\alpha}}f)(\boldsymbol \xi)|^2d\boldsymbol \xi\leq A_{\theta}(\lambda(E))^\frac{2\theta}{n} \int_{\mathbb{R}^n}\|\csc{(\alpha-\beta})\boldsymbol x\|^{2\theta}\big|(\mathfrak{F}_\beta f)(\boldsymbol x)\big|^2d\boldsymbol x,$$
which implies
\begin{equation}\label{P1eqn28}
\int_E|(\mathfrak{F}_{\alpha}f)(\boldsymbol\xi)|^2d\boldsymbol\xi\leq \frac{A_{\theta}}{|\sin(\alpha-\beta)|^{2\theta}}(\lambda(E))^\frac{2\theta}{n}\|\|\boldsymbol x\|^{\theta}(\mathfrak{F}_{\beta}f)\|^2_{L^2(\mathbb{R}^n)}.
\end{equation}
Again, for $\theta>\frac{n}{2}$
\begin{equation}\label{P1eqn29}
\int_E|(\mathfrak{F}h)(\boldsymbol\xi)|^2d\boldsymbol\xi\leq A_{\theta}\lambda(E)\|h\|^{2-\frac{n}{\theta}}_{L^2(\mathbb{R}^n)}\|\|\boldsymbol x\|^{\theta}h\|^{\frac{n}{\theta}}_{L^2(\mathbb{R}^n)}.
\end{equation}
Now,
\begin{eqnarray}\label{P1eqn30}
\|h\|_{L^2(\mathbb{R}^n)}&=&\|H\|_{L^2(\mathbb{R}^n)}\notag\\ 
&=&\left(\int_{\mathbb{R}^n}|H(\boldsymbol\xi)|^2d\boldsymbol \xi\right)^\frac{1}{2}\notag\\
&=&\left(\int_{\mathbb{R}^n}|(\mathfrak{F}_\alpha f)(\boldsymbol\xi)|^2d\boldsymbol\xi\right)^\frac{1}{2}\notag\\
&=&\|f\|_{L^2(\mathbb{R}^n)}.
\end{eqnarray}
From (\ref{P1eqn24}), (\ref{P1eqn26}), (\ref{P1eqn29}) and (\ref{P1eqn30}), we get
$$\int_E|(\mathfrak{F}_{\alpha}f)(\boldsymbol\xi)|^2d\boldsymbol\xi\leq A_{\theta}\lambda(E)\|f\|^{2-\frac{n}{\theta}}_{L^2(\mathbb{R}^n)}\left(\int_{\mathbb{R}^n}|\|\csc(\alpha-\beta)\boldsymbol x\|^{\theta}(\mathfrak{F}_{\beta}f)(\boldsymbol x)|^2d\boldsymbol x\right)^{\frac{n}{2\theta}},$$
which implies
\begin{equation}\label{P1eqn31}
\int_E|(\mathfrak{F}_{\alpha}f)(\boldsymbol\xi)|^2d\boldsymbol\xi\leq \frac{A_{\theta}}{|\sin(\alpha-\beta)|^{n}}\lambda(E)\|f\|^{2-\frac{n}{\theta}}_{L^2(\mathbb{R}^n)}\|\|\boldsymbol x\|^{\theta}(\mathfrak{F}_{\beta}f)\|^{\frac{n}{\theta}}_{L^2(\mathbb{R}^n)}.
\end{equation}
This completes the lemma.
 \end{proof}
 Before we state the main result of this subsection, we prove the following lemma.
 \begin{lemma}\label{P1lemma4.2}
 Let $\psi$ be a wavelets and $f\in L^2(\mathbb{R}^n)$. Let $E$ be a measurable subset of $\mathbb{R}^n.$ Then
 $$\int_{\mathbb{R}_{0}^n}\int_{E}|(\mathfrak{F}_{\alpha,\boldsymbol b}\{(W_{\psi,\alpha}f)(\boldsymbol b,\boldsymbol a)\})(\boldsymbol\xi)|^2d\boldsymbol\xi\frac{d\boldsymbol a}{|\boldsymbol a|_{p}^2}=\frac{C_{\psi,\alpha}}{|C_{\alpha}|^2}\int_{E}|(\mathfrak{F}_{\alpha}f)(\boldsymbol\xi)|^2d\boldsymbol\xi,$$
 where $C_{\psi,\alpha}$ is as defined in (\ref{P1eqn5}) and $\mathfrak{F}_{\alpha,\boldsymbol b}$ denotes the FrFT to be taken with  respect to the argument $\boldsymbol b.$
 \end{lemma}
 \begin{proof}
 Using lemma (\ref{P1lemma2.3}), we have
 \begin{eqnarray*}
 \int_{\mathbb{R}_{0}^n}\int_E|(\mathfrak{F}_{\alpha,\boldsymbol b}\{(W_{\psi,\alpha}f)(\boldsymbol b,\boldsymbol a)\})(\boldsymbol\xi)|^2d\boldsymbol\xi\frac{d\boldsymbol a}{|\boldsymbol a|_{p}^2}&=&\int_{\mathbb{R}_{0}^n}\int_E\frac{|\boldsymbol a|_{p}}{|C_{\alpha}|^2}\left|\left(\mathfrak{F}_{\alpha}\left\{e^{-\frac{i}{2}\|\cdot\|^2\cot\alpha}\psi\right\}\right)(\boldsymbol a\boldsymbol\xi)\right|^2|(\mathfrak{F}_{\alpha}f)(\boldsymbol\xi)|^2 d\boldsymbol\xi\frac{d\boldsymbol a}{|\boldsymbol a|_{p}^2}\\
 &=&\frac{1}{|C_{\alpha}|^2}\int_E|(\mathfrak{F}_{\alpha}f)(\boldsymbol\xi)|^2\left(\int_{\mathbb{R}_{0}^n}\left|\left(\mathfrak{F}_{\alpha}\left\{e^{-\frac{i}{2}\|\cdot\|^2\cot\alpha}\psi(\cdot)\right\}\right)(\boldsymbol a\boldsymbol\xi)\right|^2\frac{1}{|\boldsymbol a|_{p}}d\boldsymbol a\right)d\boldsymbol\xi\\
 &=&\frac{C_{\psi,\alpha}}{|C_{\alpha}|^2}\int_E|(\mathfrak{F}_{\alpha}f)(\boldsymbol\xi)|^2d\boldsymbol\xi.
 \end{eqnarray*}
This completes the proof.
 \end{proof}
We now prove the main result of this subsection.
 \begin{theorem}\label{P1theo4.1}
(Local uncertainty Inequality for CFrWT) Let $\psi$ be a wavelet and $\theta\neq\frac{n}{2}$ be a given positive number. Then there exists a constant $A_{\theta}$ such that for all measurable subset $E$ of $\mathbb{R}^n$ and all $f\in L^2(\mathbb{R}^n),$
\begin{equation*}
\begin{array}{l}
\hspace{-1.5cm}\displaystyle\int_E|(\mathfrak{F}_{\alpha}f)(\boldsymbol\xi)|^2d\boldsymbol\xi\\
\leq\frac{A_{\theta}{|C_{\alpha}|^2}}{|\sin(\alpha-\beta)|^{2\theta}C_{\psi,\alpha}}(\lambda(E))^\frac{2\theta}{n}\displaystyle\int_{\mathbb{R}_{0}^n}\int_{\mathbb{R}^n}\|\boldsymbol x\|^{2\theta}|(\mathfrak{F}_{\beta,\boldsymbol b}\{(W_{\psi,\alpha}f)(\boldsymbol b,\boldsymbol a)\})(\boldsymbol x)|^2d\boldsymbol x\frac{d\boldsymbol a}{|\boldsymbol a|^2_{p}},\ \mbox{if}\ 0<\theta<\frac{n}{2}
\end{array}
\end{equation*} 
 and
\begin{equation*}
\begin{array}{l}
 \hspace{-1cm}\displaystyle\int_E|(\mathfrak{F}_{\alpha}f)(\boldsymbol\xi)|^2d\boldsymbol\xi\\[6pt] \leq\frac{A_{\theta}|C_{\alpha}|^2}{|\sin(\alpha-\beta)|^{n}C_{\psi,\alpha}}\lambda(E)\displaystyle\int_{\mathbb{R}_{0}^n}\|(W_{\psi,\alpha}f)(\cdot,\boldsymbol a)\|^{2-\frac{n}{\theta}}_{L^2(\mathbb{R}^n)}\|\|\boldsymbol x\|^{\theta}(\mathfrak{F}_{\beta,\boldsymbol b}\{(W_{\psi,\alpha}f)(\boldsymbol b,\boldsymbol a)\})\|^{\frac{n}{\theta}}_{L^2(\mathbb{R}^n)}\frac{d\boldsymbol a}{|\boldsymbol a|^2_{p}},\ \mbox{if}\ \theta>\frac{n}{2}.
 \end{array}
\end{equation*} 
 \end{theorem}
\begin{proof}
Replacing $f$ by $(W_{\psi,\alpha}f)(\boldsymbol b,\boldsymbol a)$ in (\ref{P1eqn28}), we obtain
$$\int_E|(\mathfrak{F}_{\alpha,\boldsymbol b}\{(W_{\psi,\alpha}f)(\boldsymbol b,\boldsymbol a)\})(\boldsymbol\xi)|^2d\boldsymbol\xi\leq \frac{A_{\theta}}{|\sin(\alpha-\beta)|^{2\theta}}(\lambda(E))^\frac{2\theta}{n}\|\|\boldsymbol x\|^{\theta}(\mathfrak{F}_{\beta,\boldsymbol b}\{(W_{\psi,\alpha}f)(\boldsymbol b,\boldsymbol a)\})\|^2_{L^2(\mathbb{R}^n)}.$$
Integrating both sides with respect to the measure $\frac{d\boldsymbol a}{|\boldsymbol a|^2_{p}},$ we obtain
$$\int_{\mathbb{R}_{0}^n}\int_E|(\mathfrak{F}_{\alpha,\boldsymbol b}\{(W_{\psi,\alpha}f)(\boldsymbol b,\boldsymbol a)\})(\boldsymbol\xi)|^2\frac{d\boldsymbol a}{|\boldsymbol a|^2_{p}}d\boldsymbol\xi\leq \frac{A_{\theta}}{|\sin(\alpha-\beta)|^{2\theta}}(\lambda(E))^\frac{2\theta}{n}\int_{\mathbb{R}_{0}^n}\|\|\boldsymbol x\|^{\theta}(\mathfrak{F}_{\beta,\boldsymbol b}\{(W_{\psi,\alpha}f)(\boldsymbol b,\boldsymbol a)\})\|^2_{L^2(\mathbb{R}^n)}\frac{d\boldsymbol a}{|\boldsymbol a|^2_{p}}.$$
Using lemma \ref{P1lemma4.2}, we get
$$\frac{C_{\psi,\alpha}}{|C_{\alpha}|^2}\int_E|(\mathfrak{F}_{\alpha}f)(\boldsymbol\xi)|^2d\boldsymbol\xi\leq\frac{A_{\theta}}{|\sin(\alpha-\beta)|^{2\theta}}(\lambda(E))^\frac{2\theta}{n}\int_{\mathbb{R}_{0}^n}\int_{\mathbb{R}^n}|\|\boldsymbol x\|^{\theta}(\mathfrak{F}_{\beta,\boldsymbol b}\{(W_{\psi,\alpha}f)(\boldsymbol b,\boldsymbol a)\})(\boldsymbol x)|^2d\boldsymbol x\frac{d\boldsymbol a}{|\boldsymbol a|^2_{p}},$$
which implies 
$$\int_E|(\mathfrak{F}_{\alpha}f)(\boldsymbol\xi)|^2d\boldsymbol\xi\leq\frac{A_{\theta}{|C_{\alpha}|^2}}{|\sin(\alpha-\beta)|^{2\theta}C_{\psi,\alpha}}(\lambda(E))^\frac{2\theta}{n}\int_{\mathbb{R}_{0}^n}\int_{\mathbb{R}^n}\|\boldsymbol x\|^{2\theta}|(\mathfrak{F}_{\beta,\boldsymbol b}\{(W_{\psi,\alpha}f)(\boldsymbol b,\boldsymbol a)\})(\boldsymbol x)|^2d\boldsymbol x\frac{d\boldsymbol a}{|\boldsymbol a|^2_{p}}.$$
Replacing $f$ by $(W_{\psi,\alpha}f)(\boldsymbol b,\boldsymbol a)$ in equation (\ref{P1eqn31}), we get
$$\int_E|(\mathfrak{F}_{\alpha,\boldsymbol b}\{(W_{\psi,\alpha})(\boldsymbol b,\boldsymbol a)\})(\boldsymbol\xi)|^2d\boldsymbol\xi\leq \frac{A_{\theta}}{|\sin(\alpha-\beta)|^{n}}\lambda(E)\|(W_{\psi,\alpha}f)(\cdot,\boldsymbol a)\|^{2-\frac{n}{\theta}}_{L^2(\mathbb{R}^n)}\|\|\boldsymbol x\|^{\theta}(\mathfrak{F}_{\beta,\boldsymbol b}\{(W_{\psi,\alpha}f)(\boldsymbol b,\boldsymbol a)\})\|^{\frac{n}{\theta}}_{L^2(\mathbb{R}^n)}.$$
Integrating both sides with respect to the measure $\frac{d\boldsymbol a}{|\boldsymbol a|^2_{p}}$ and applying lemma \ref{P1lemma4.2}, yields
$$\int_E|(\mathfrak{F}_{\alpha}f)(\boldsymbol\xi)|^2d\boldsymbol\xi\leq\frac{A_{\theta}|C_{\alpha}|^2}{|\sin(\alpha-\beta)|^{n}C_{\psi,\alpha}}\lambda(E)\int_{\mathbb{R}_{0}^n}\|(W_{\psi,\alpha}f)(\cdot,\boldsymbol a)\|^{2-\frac{n}{\theta}}_{L^2(\mathbb{R}^n)}\|\|\boldsymbol x\|^{\theta}(\mathfrak{F}_{\beta,\boldsymbol b}\{(W_{\psi,\alpha}f)(\boldsymbol b,\boldsymbol a)\})\|^{\frac{n}{\theta}}_{L^2(\mathbb{R}^n)}\frac{d\boldsymbol a}{|\boldsymbol a|^2_{p}}.$$
This completes the proof.
\end{proof} 
The corollary below is a Local uncertainty inequality for classical wavelet transform in $n$-dimensions.
\begin{corollary}
Let $\psi$ be a wavelet and $\theta$ be a given positive number. Then there exists a constant $A_{\theta}$ such that for all measurable subset $E$ of $\mathbb{R}^n$ and all $f\in L^2(\mathbb{R}^n),$
\begin{equation*}
\begin{array}{l}
\hspace{-1.5cm}\displaystyle\int_E|(\mathfrak{F}f)(\boldsymbol\xi)|^2d\boldsymbol\xi\\
\leq\frac{A_{\theta}}{(2\pi)^nC_{\psi}}(\lambda(E))^\frac{2\theta}{n}\displaystyle\int_{\mathbb{R}_{0}^n}\int_{\mathbb{R}^n}\|\boldsymbol x\|^{2\theta}|(W_{\psi}f)(\boldsymbol x,\boldsymbol a)|^2d\boldsymbol x\frac{d\boldsymbol a}{|\boldsymbol a|^2_{p}},\ \mbox{if}\ 0<\theta<\frac{n}{2}
\end{array}
\end{equation*} 
 and
\begin{equation*}
\begin{array}{l}
 \hspace{-1cm}\displaystyle\int_E|(\mathfrak{F}f)(\boldsymbol\xi)|^2d\boldsymbol\xi\\[6pt] \leq\frac{A_{\theta}}{(2\pi)^nC_{\psi}}\lambda(E)\displaystyle\int_{\mathbb{R}_{0}^n}\|(W_{\psi}f)(\cdot,\boldsymbol a)\|^{2-\frac{n}{\theta}}_{L^2(\mathbb{R}^n)}\|\|\boldsymbol x\|^{\theta}\{(W_{\psi}f)(\boldsymbol x,\boldsymbol a)\}\|^{\frac{n}{\theta}}_{L^2(\mathbb{R}^n)}\frac{d\boldsymbol a}{|\boldsymbol a|^2_{p}},\ \mbox{if}\ \theta>\frac{n}{2}.
 \end{array}
\end{equation*} 
\end{corollary}
\section{CFrWT on Morrey Space}
In this section we take our wavelets to be integrable and study various properties of CFrWT on Morrey space. The main result of this section will be the boundedness of the CFrWT and the estimate of $L_{M}^{p,\nu}(\mathbb{R}^n)$-distance of the CFrWT  of two argument functions with respect to two wavelets. We shall begin with the definition of the Morrey spaces.
\begin{defn}
Let $1\leq p<\infty$ and $\nu\geq 0.$ The Morrey space, denoted by $L^{p,\nu}_{M}(\mathbb{R}^n),$ is defined as
$$ L^{p,\nu}_{M}(\mathbb{R}^n)=\left\{f\in L^p(\mathbb{R}^n): \sup_{\boldsymbol x\in\mathbb{R}^n,r>0}\frac{1}{r^{\nu}}\int_{B(\boldsymbol x,r)}|f(\boldsymbol t)|^pd\boldsymbol t<\infty\right\},$$ which is a Banach space with respect to the norm
\begin{equation}\label{P1eqn32}
\|f\|_{L^{p,\nu}_{M}(\mathbb{R}^n)}=\sup_{\boldsymbol x\in\mathbb{R}^n,r>0}\left(\frac{1}{r^{\nu}}\int_{B(\boldsymbol x,r)}|f(\boldsymbol t)|^p\right)^{\frac{1}{p}}.
\end{equation}
\end{defn}
\begin{lemma}\label{P1lemma4.3}
For a wavelet $\psi$ and a function $f\in L^{1,\nu}_{M}(\mathbb{R}^n)$ the CFrWT of $f,$ i.e., $(W_{\psi,\alpha}f)(\cdot,\boldsymbol a)$ is in $L^1(\mathbb{R}^n).$
\end{lemma}
\begin{proof}
By the definition of CFrWT, we have
$$(W_{\psi,\alpha}f)(\boldsymbol b,\boldsymbol a)=\int_{\mathbb{R}^n} f(\boldsymbol t)\overline{\psi_{\boldsymbol a,\boldsymbol b,\alpha}(\boldsymbol t)}d\boldsymbol t.$$
Now,
\begin{eqnarray*}
|(W_{\psi,\alpha}f)(\boldsymbol b,\boldsymbol a)|\leq \int_{\mathbb{R}^n}|f(\boldsymbol t)||\psi_{\boldsymbol a,\boldsymbol b,\alpha}(\boldsymbol t)|d\boldsymbol t.
\end{eqnarray*}
Using equation (\ref{P1eqnC}), we get
\begin{align}\label{P1eqn33}
|(W_{\psi,\alpha}f)(\boldsymbol b,\boldsymbol a)|&\leq  \frac{1}{\sqrt{|\boldsymbol a|_{p}}}\int_{\mathbb{R}^n}|f(\boldsymbol t)||\psi\left(\frac{\boldsymbol t-\boldsymbol b}{\boldsymbol a}\right)|d\boldsymbol t\notag\\
&= \sqrt{|\boldsymbol a|_{p}}\int_{\mathbb{R}^n}|f(\boldsymbol a\boldsymbol x+\boldsymbol b)||\psi(\boldsymbol x)|d\boldsymbol x.
\end{align}
This implies,
\begin{align*}
\int_{\mathbb{R}^n}|(W_{\psi,\alpha}f)(\boldsymbol b,\boldsymbol a)|d\boldsymbol b&\leq \sqrt{|\boldsymbol a|_{p}}\int_{\mathbb{R}^n}\int_{\mathbb{R}^n}|f(\boldsymbol a\boldsymbol x+\boldsymbol b)||\psi(\boldsymbol x)|d\boldsymbol x d\boldsymbol b\\
&=\sqrt{|\boldsymbol a|_{p}}\int_{\mathbb{R}^n}|\psi(\boldsymbol x)|\int_{\mathbb{R}^n}|f(\boldsymbol a\boldsymbol x+\boldsymbol b)| d\boldsymbol b d\boldsymbol x\\
&= \sqrt{|\boldsymbol a|_{p}}\|\psi\|_{L^1(\mathbb{R}^n)}\|f\|_{L^1(\mathbb{R}^n)}.
\end{align*}
Thus, it follow that $(W_{\psi,\alpha}f)(\cdot,\boldsymbol a)$ is in $L^1(\mathbb{R}^n).$
\end{proof}
\begin{theorem}\label{P1theo4.2}
For any $\boldsymbol a\in\mathbb{R}^n_{0},$ the operator $W_{\psi,\alpha}:L^{1,\nu}_{M}(\mathbb{R}^n)\rightarrow L^{1,\nu}_{M}(\mathbb{R}^n)$ defined by $f\rightarrow (W_{\psi,\alpha}f)(\cdot,\boldsymbol a)$ is bounded. Furthermore,
$$\|(W_{\psi,\alpha}f)(\cdot,\boldsymbol a)\|_{L^{1,\nu}_{M}(\mathbb{R}^n)}\leq\sqrt{|\boldsymbol a|_{p}}\|\psi\|_{L^1(\mathbb{R}^n)}\|f\|_{L^{1,\nu}_{M}(\mathbb{R}^n)}.$$
\end{theorem}
\begin{proof}
By the definition of $L^{1,\nu}_{M}(\mathbb{R}^n)$-norm, we have
\begin{equation}\label{P1eqn34}
\|(W_{\psi,\alpha}f)(\cdot,\boldsymbol a)\|_{L^{1,\nu}_{M}(\mathbb{R}^n)}=\sup_{\boldsymbol x\in\mathbb{R}^n,r>0}\left(\frac{1}{r^{\nu}}\int_{B(\boldsymbol x,r)}|(W_{\psi,\alpha}f)(\boldsymbol b,\boldsymbol a)|d\boldsymbol b\right).
\end{equation}
Now using equation (\ref{P1eqn33}), we get
\begin{eqnarray}\label{P1eqn35}
\frac{1}{r^{\nu}}\int_{B(\boldsymbol x,r)}|(W_{\psi,\alpha}f)(\boldsymbol b,\boldsymbol a)|d\boldsymbol b&\leq &\frac{\sqrt{|\boldsymbol a|_{p}}}{r^{\nu}}\int_{B(\boldsymbol x,r)}\left(\int_{\mathbb{R}^n}|f(\boldsymbol a\boldsymbol u+\boldsymbol b)||\psi(\boldsymbol u)|d\boldsymbol u\right)d\boldsymbol b\notag\\
&= &\frac{\sqrt{|\boldsymbol a|_{p}}}{r^{\nu}}\int_{\mathbb{R}^n}|\psi(\boldsymbol u)|\left(\int_{B(\boldsymbol x,r)}|f(\boldsymbol a\boldsymbol u+\boldsymbol b)|d\boldsymbol b\right)d\boldsymbol u\notag\\
&=&\sqrt{|\boldsymbol a|_{p}}\int_{\mathbb{R}^n}|\psi(\boldsymbol u)|\left(\frac{1}{r^{\nu}}\int_{B(\boldsymbol a\boldsymbol u+\boldsymbol x,r)}|f(\boldsymbol z)|d\boldsymbol z\right)d\boldsymbol u
\end{eqnarray}
Since 
\begin{eqnarray}\label{P1eqn36}
\frac{1}{r^{\nu}}\int_{B(\boldsymbol a\boldsymbol u+\boldsymbol x,r)}|f(\boldsymbol z)|d\boldsymbol z &\leq& \sup_{{\boldsymbol y\in\mathbb{R}^n,r>0}}\left(\frac{1}{r^{\nu}}\int_{B(\boldsymbol y,r)}|f(\boldsymbol z)|d\boldsymbol z \right)\notag\\
&=&\|f\|_{L^{1,\nu}_{M}(\mathbb{R}^n).}
\end{eqnarray}
From equation (\ref{P1eqn35}) and equation (\ref{P1eqn36}), we have
$$\frac{1}{r^{\nu}}\int_{B(\boldsymbol x,r)}|(W_{\psi,\alpha}f)(\boldsymbol b,\boldsymbol a)|d\boldsymbol b\leq\sqrt{|\boldsymbol a|_{p}}\|f\|_{L^{1,\nu}_{M}(\mathbb{R}^n)}\|\psi\|_{L^1(\mathbb{R}^n)},$$
which gives
\begin{equation}
\label{P1eqn37}
\sup_{\boldsymbol x\in\mathbb{R}^n,r>0}\left(\frac{1}{r^{\nu}}\int_{B(\boldsymbol x,r)}|(W_{\psi,\alpha}f)(\boldsymbol b,\boldsymbol a)|d\boldsymbol b\right)\leq\sqrt{|\boldsymbol a|_{p}}\|f\|_{L^{1,\nu}_{M}(\mathbb{R}^n)}\|\psi\|_{L^1(\mathbb{R}^n)}.
\end{equation}
By virtue of equation (\ref{P1eqn34}), equation (\ref{P1eqn37}) and lemma \ref{P1lemma4.3}, it follows that $(W_{\phi,\alpha}f)(\cdot,\boldsymbol a)\in L^{1,\nu}_{M}(\mathbb{R}^n)$ and 
$$\|(W_{\psi,\alpha}f)(\cdot,\boldsymbol a)\|_{L^{1,\nu}_{M}(\mathbb{R}^n)}\leq\sqrt{|\boldsymbol a|_{p}}\|\psi\|_{L^1(\mathbb{R}^n)}\|f\|_{L^{1,\nu}_{M}(\mathbb{R}^n)}.$$
This proves the theorem.
\end{proof}
\begin{corollary}\label{P1corollary4.1}
Let $\boldsymbol a\in\mathbb{R}^n_{0}$, $\psi$ be a wavelet and $f\in L^{1,\nu}_{M}(\mathbb{R}^n),$ then $$\|(W_{\psi,\alpha}f)(\cdot,\boldsymbol a)\|_{L^{1,\nu}_{M}(\mathbb{R}^n)}=O(\sqrt{|\boldsymbol a|_{p}}).$$ 
\end{corollary}
\begin{theorem}\label{P1theo4.3}
Let $\phi,\psi$ be two wavelets and $f\in L^{1,\nu}_{M}(\mathbb{R}^n).$ Then
$$\|(W_{\phi,\alpha}f)(\cdot,\boldsymbol a)-(W_{\psi,\alpha}f)(\cdot,\boldsymbol a)\|_{L^{1,\nu}_{M}(\mathbb{R}^n)}\leq \sqrt{|\boldsymbol a|_{p}}\|f\|_{L^{1,\nu}_{M}(\mathbb{R}^n)}\|\phi-\psi\|_{L^1(\mathbb{R}^n)}.$$ 
\end{theorem}
\begin{proof}
We have, 
\begin{equation}{\label{P1eqn38}}
\|(W_{\phi,\alpha}f)(\cdot,\boldsymbol a)-(W_{\psi,\alpha}f)(\cdot,\boldsymbol a)\|_{L^{1,\nu}_{M}(\mathbb{R}^n)}=\sup_{\boldsymbol x\in\mathbb{R}^n,r>0}\frac{1}{r^{\nu}}\int_{B(\boldsymbol x,r)}\left|\int_{\mathbb{R}^n}\big(f(\boldsymbol t)\overline{\phi_{\boldsymbol a,\boldsymbol y,\alpha}(\boldsymbol t)}-f(\boldsymbol t)\overline{\psi_{\boldsymbol a,\boldsymbol y,\alpha}(\boldsymbol t)}\big)d\boldsymbol t\right| d\boldsymbol y.
\end{equation}
Now,
\begin{eqnarray*}
\frac{1}{r^{\nu}}\int_{B(\boldsymbol x,r)}\left|\int_{\mathbb{R}^n}\big(f(\boldsymbol t)\overline{\phi_{\boldsymbol a,\boldsymbol y,\alpha}(\boldsymbol t)}-f(\boldsymbol t)\overline{\psi_{\boldsymbol a,\boldsymbol y,\alpha}(\boldsymbol t)}\big)d\boldsymbol t\right| d\boldsymbol y&\leq & \frac{1}{r^{\nu}}\int_{B(\boldsymbol x,r)}\int_{\mathbb{R}^n}\left|f(\boldsymbol t)\big(\phi_{\boldsymbol a,\boldsymbol y,\alpha}(\boldsymbol t)-\psi_{\boldsymbol a,\boldsymbol y,\alpha}(\boldsymbol t)\big)\right|d\boldsymbol t d\boldsymbol y\\
&=& \frac{1}{r^{\nu}}\int_{B(\boldsymbol x,r)}\bigg(\int_{\mathbb{R}^n}\frac{1}{\sqrt{|\boldsymbol a|_{p}}}|f(\boldsymbol t)|\left|\phi\bigg(\frac{\boldsymbol t-\boldsymbol y}{\boldsymbol a}\bigg)-\psi\left(\frac{\boldsymbol t-\boldsymbol y}{\boldsymbol a}\right)\right|d\boldsymbol t\bigg) d\boldsymbol y\\
&=& \frac{\sqrt{|\boldsymbol a|_{p}}}{r^{\nu}}\int_{B(\boldsymbol x,r)}\left(\int_{\mathbb{R}^n}|f(\boldsymbol a\boldsymbol u+\boldsymbol y)||\phi(\boldsymbol u)-\psi(\boldsymbol u)|d\boldsymbol u\right) d\boldsymbol y\\
&=&\frac{\sqrt{|\boldsymbol a|_{p}}}{r^{\nu}}\int_{\mathbb{R}^n}|\phi(\boldsymbol u)-\psi(\boldsymbol u)|\left(\int_{B(\boldsymbol x,r)}|f(\boldsymbol a\boldsymbol u+\boldsymbol y)|d\boldsymbol y\right)d\boldsymbol u\\
&=&\sqrt{|\boldsymbol a|_{p}}\int_{\mathbb{R}^n}|\phi(\boldsymbol u)-\psi(\boldsymbol u)|\left(\frac{1}{r^{\nu}}\int_{B(\boldsymbol a\boldsymbol u+\boldsymbol x,r)}|f(\boldsymbol z)|d\boldsymbol z\right)d\boldsymbol u.\\
\end{eqnarray*}
Using equation (\ref{P1eqn36}), we get
$$\frac{1}{r^{\nu}}\int_{B(\boldsymbol x,r)}\left|\int_{\mathbb{R}^n}\big(f(\boldsymbol t)\overline{\phi_{\boldsymbol a,\boldsymbol y,\alpha}(\boldsymbol t)}-f(\boldsymbol t)\overline{\psi_{\boldsymbol a,\boldsymbol y,\alpha}(\boldsymbol t)}\big)d\boldsymbol t\right| d\boldsymbol y\leq\sqrt{|\boldsymbol a|_{p}}\|f\|_{L^{1,\nu}_{M}(\mathbb{R}^n)}\int_{\mathbb{R}^n}|\phi(\boldsymbol u)-\psi(\boldsymbol u)|d\boldsymbol u.$$

Therefore,
\begin{equation}\label{P1eqn39}
\sup_{\boldsymbol x\in\mathbb{R}^n,r>0}\frac{1}{r^{\nu}}\int_{B(\boldsymbol x,r)}\left|\int_{\mathbb{R}^n}\big(f(\boldsymbol t)\overline{\phi_{\boldsymbol a,\boldsymbol y,\alpha}(\boldsymbol t)}-f(\boldsymbol t)\overline{\psi_{\boldsymbol a,\boldsymbol y,\alpha}(\boldsymbol t)}\big)d\boldsymbol t\right| d\boldsymbol y\leq\sqrt{|\boldsymbol a|_{p}}\|f\|_{L^{1,\nu}_{M}(\mathbb{R}^n)}\|\phi-\psi\|_{L^1(\mathbb{R}^n)}.
\end{equation}
From equation (\ref{P1eqn38}) and equation (\ref{P1eqn39}), the theorem follows immediately.
\end{proof}
\begin{theorem}\label{P1theo4.4}
Let $\psi$ be a wavelet and $f,g\in L^{1,\nu}_{M}(\mathbb{R}^n).$ Then
$$\|(W_{\psi,\alpha}f)(\cdot,\boldsymbol a)-(W_{\psi,\alpha}g)(\cdot,\boldsymbol a)\|_{L^{1,\nu}_{M}(\mathbb{R}^n)}\leq \sqrt{|\boldsymbol a|_{p}}\|f-g\|_{L^{1,\nu}_{M}(\mathbb{R}^n)}\|\psi\|_{L^1(\mathbb{R}^n)}.$$ 
\end{theorem}
\begin{proof}
We have, 
\begin{equation}\label{P1eqn40}
\|(W_{\psi,\alpha}f)(\cdot,\boldsymbol a)-(W_{\psi,\alpha}g)(\cdot,\boldsymbol a)\|_{L^{1,\nu}_{M}(\mathbb{R}^n)}=\sup_{\boldsymbol x\in\mathbb{R}^n,r>0}\frac{1}{r^{\nu}}\int_{B(\boldsymbol x,r)}\left|\int_{\mathbb{R}^n}\big(f(\boldsymbol t)\overline{\psi_{\boldsymbol a,\boldsymbol y,\alpha}(\boldsymbol t)}-g(\boldsymbol t)\overline{\psi_{\boldsymbol a,\boldsymbol y,\alpha}(\boldsymbol t)}\big)d\boldsymbol t\right| d\boldsymbol y.
\end{equation}
Now,
\begin{eqnarray*}
\frac{1}{r^{\nu}}\int_{B(\boldsymbol x,r)}\left|\int_{\mathbb{R}^n}\big(f(\boldsymbol t)\overline{\psi_{\boldsymbol a,\boldsymbol y,\alpha}(\boldsymbol t)}-g(\boldsymbol t)\overline{\psi_{\boldsymbol a,\boldsymbol y,\alpha}(\boldsymbol t)}\big)d\boldsymbol t\right| d\boldsymbol y&\leq & \frac{1}{r^{\nu}}\int_{B(\boldsymbol x,r)}\int_{\mathbb{R}^n}|f(\boldsymbol t)-g(\boldsymbol t)||\psi_{\boldsymbol a,\boldsymbol y,\alpha}(\boldsymbol t)|d\boldsymbol t d\boldsymbol y\\
&=& \frac{1}{r^{\nu}}\int_{B(\boldsymbol x,r)}\left(\int_{\mathbb{R}^n}\frac{1}{\sqrt{|\boldsymbol a|_{p}}}|f(\boldsymbol t)-g(\boldsymbol t)|\left|\psi\left(\frac{\boldsymbol t-\boldsymbol y}{\boldsymbol a}\right)\right|d\boldsymbol t\right) d\boldsymbol y\\
&=& \frac{\sqrt{|\boldsymbol a|_{p}}}{r^{\nu}}\int_{B(\boldsymbol x,r)}\left(\int_{\mathbb{R}^n}|f(\boldsymbol a\boldsymbol u+\boldsymbol y)-g(\boldsymbol a\boldsymbol u+\boldsymbol y)||\psi(\boldsymbol u)|d\boldsymbol u\right) d\boldsymbol y\\
&=&\frac{\sqrt{|\boldsymbol a|_{p}}}{r^{\nu}}\int_{\mathbb{R}^n}|\psi(\boldsymbol u)|\left(\int_{B(\boldsymbol x,r)}|f(\boldsymbol a\boldsymbol u+\boldsymbol y)-g(\boldsymbol a\boldsymbol u+\boldsymbol y)|d\boldsymbol y\right)d\boldsymbol u\\
&=&\sqrt{|\boldsymbol a|_{p}}\int_{\mathbb{R}^n}|\psi(\boldsymbol u)|\left(\frac{1}{r^{\nu}}\int_{B(\boldsymbol a\boldsymbol u+\boldsymbol x,r)}|f(\boldsymbol z)-g(\boldsymbol z)|d\boldsymbol z\right)d\boldsymbol u\\
&\leq &\sqrt{|\boldsymbol a|_{p}}\|f-g\|_{L^{1,\nu}_{M}(\mathbb{R}^n)}\int_{\mathbb{R}^n}|\psi(\boldsymbol u)|d\boldsymbol u.\\
\end{eqnarray*}
Therefore,
\begin{equation}\label{P1eqn41}
\sup_{\boldsymbol x\in\mathbb{R}^n,r>0}\frac{1}{r^{\nu}}\int_{B(\boldsymbol x,r)}\left|\int_{\mathbb{R}^n}\big(f(\boldsymbol t)\overline{\psi_{\boldsymbol a,\boldsymbol y,\alpha}(\boldsymbol t)}-g(\boldsymbol t)\overline{\psi_{\boldsymbol a,\boldsymbol y,\alpha}(\boldsymbol t)}\big)d\boldsymbol t\right| d\boldsymbol y\leq\sqrt{|\boldsymbol a|_{p}}\|f-g\|_{L^{1,\nu}_{M}(\mathbb{R}^n)}\int_{\mathbb{R}^n}|\psi(\boldsymbol u)|d\boldsymbol u.
\end{equation}
The theorem follows from equation (\ref{P1eqn40}) and equation (\ref{P1eqn41}).
\end{proof}
We now give the main result of this section.
\begin{theorem}\label{P1theo4.5}
Let $\phi,\psi$ be two wavelets and $f,g\in L^{1,\nu}_{M}(\mathbb{R}^n).$ Then 
$$\|(W_{\phi,\alpha}f)(\cdot,\boldsymbol a)-(W_{\psi,\alpha}g)(\cdot,\boldsymbol a)\|_{L^{1,\nu}_{M}(\mathbb{R}^n)}\leq \sqrt{|\boldsymbol a|_{p}}\Big(\|f\|_{L^{1,\nu}_{M}(\mathbb{R}^n)}\|\phi-\psi\|_{L^1(\mathbb{R}^n)}+\|f-g\|_{L^{1,\nu}_{M}(\mathbb{R}^n)}\|\psi\|_{L^1(\mathbb{R}^n)}\Big).$$
\end{theorem}
\begin{proof}
We have
\begin{align*}
\|(W_{\phi,\alpha}f)(\cdot,\boldsymbol a)-(W_{\psi,\alpha}g)(\cdot,\boldsymbol a)&\|_{L^{1,\nu}_{M}(\mathbb{R}^n)}\\
&\leq\|(W_{\phi,\alpha}f)(\cdot,\boldsymbol a)-(W_{\psi,\alpha}f)(\cdot,\boldsymbol a)\|_{L^{1,\nu}_{M}(\mathbb{R}^n)}+\|(W_{\psi,\alpha}f)(\cdot,\boldsymbol a)-(W_{\psi,\alpha}g)(\cdot,\boldsymbol a)\|_{L^{1,\nu}_{M}(\mathbb{R}^n)}.
\end{align*}
Using theorem \ref{P1theo4.3} and theorem \ref{P1theo4.4}, the proof follows.
\end{proof}
\section{Conclusion}
 In this paper, we have introduced the CFrWT in $n$-dimensions, which is a  generalization of classical wavelet transform. This paper also serves as a generalization of the theory of CFrWT studied in \cite{bahri2017logarithmic,prasad2015continuous,prasad2014generalized}. We have derived the inner product relation of CFrWTs, of two argument functions, associated with two different wavelets.  The reconstruction formula for the CFrWT depending on two wavelets along with the reproducing kernel function, involving two wavelets, for the image space of CFrWT are also established. In section 4, using the idea of Guanlei et al. (\cite{guanlei2009logarithmic}) we have derived  Heisenberg's uncertainty and Local uncertainty inequalties in two fractional Fourier transform domain in $n$-dimensions, followed by the inequalities for CFrWT. Finally, in the last section of this paper we have studied the boundedness and approximation properties of CFrWT on Morrey space.
\bibliography{MasterCFrWT}
\bibliographystyle{plain}

\end{document}